\documentclass[12pt]{article}
\usepackage{amssymb,latexsym,amsmath,mathrsfs,color,csquotes}
\numberwithin{equation}{section}

\newtheorem{theorem}{Theorem}[section]
\newtheorem{lemma}{Lemma}[section]
\newtheorem{definition}{Definition}[section]

\newtheorem{remark}{Remark}[section]
\newtheorem{example}{Example}[section]

\newtheorem{proposition}{Proposition}[section]
\usepackage{enumerate}
\usepackage[T1]{fontenc}
\usepackage[dvips,pdftex
]{graphicx}
\usepackage{color} 
\usepackage{verbatim}
\usepackage[colorlinks=true, urlcolor=blue,linkcolor=blue, citecolor=blue]{hyperref}

\newtheorem{thm}{Theorem}[section]

\newenvironment{proof}[1][Proof]{\par\noindent\textbf{#1.} }
{\hfill~$\square$\\\medskip}

\def\aa{\boldsymbol{\alpha}}

\def\bbe{{\mathbb E}}

\def\bbn{{\mathbb N}}

\def\cala{{\cal A}}

\def\var{\mathop{\rm Var}}
\def\Var{\mathop{\rm Var}}

\numberwithin{equation}{section} \allowdisplaybreaks

\allowdisplaybreaks

\begin{document}
\title{A Central Limit Theorem for the Length of the Longest Common 
	Subsequences in Random Words}
\author{Christian Houdr\'e\footnote{School of Mathematics, Georgia
Institute of Technology, Atlanta, Georgia, 30332-0160, USA, 
houdre@math.gatech.edu. }  \and \"{U}mit I\c{s}lak\footnote{
Bogazici University, Faculty of Arts and Science, Department of
Mathematics, 34342, Bebek-Istanbul, Turkey, umitislak@gmail.com. 
}		
} \maketitle \vspace{0.5cm}

\begin{abstract}
Let $(X_i)_{i \geq 1}$ and $(Y_i)_{i\geq1}$ be two independent
sequences of independent identically distributed (iid) random variables
taking their values in a common finite alphabet and having the same law. 
Let $LC_n$ be the length of the longest common subsequences of the two 
random words $X_1\cdots X_n$ and $Y_1\cdots Y_n$. Under a lower bound assumption  
on the order of its variance, $LC_n$ is shown to satisfy a central 
limit theorem.   This is in contrast to the limiting distribution of 
the length of the longest common subsequences in two independent uniform
random permutations of $\{1, \dots, n\}$, which is shown to be the
Tracy-Widom distribution.
\end{abstract}

\section{Introduction}

We explore here the asymptotic behavior, in law, of the length of the
longest common subsequences of two random words.  Although the study of this length is 
decade-old, and 
extensive from an algorithmic point of view, in various
disciplines such as, computer science, bioinformatics, or  
statistical physics, its mathematically rigorous results are rather sparse.  Below, we obtain the first 
result on the limiting law of this length, when properly centered and scaled.

To begin with, let us present our
framework.  Throughout, let $X=(X_i)_{i\ge 1}$ and $Y=(Y_i)_{i\ge 1}$ be two
infinite sequences whose coordinates take their values in
$\cala_m\!=\!\{\!\aa_1, \aa_2, \dots, \aa_m\!\}$, a finite alphabet
of size $m$.  Next, let $LC_n$ be the length of the longest common subsequences (LCSs) of 
the random words $X_1\cdots X_n$ and $Y_1\cdots Y_n$, i.e., $LC_n$ is the maximal 
integer $k\in \{1,\dots, n\}$, such that there exist $1\le i_1<
\cdots < i_k\le n$ and $1\le j_1<\cdots <j_k\le n$, for which:
$$
X_{i_s}=Y_{j_s}, \quad \quad {\rm for \quad all} \quad \quad
s=1,2,\dots, k.
$$

\noindent As well known, $LC_n$ is a measure of the similarity/dissimilarity of the two
words/strings which is often used in pattern matching, e.g., in computer science the 
edit (or Levenshtein) distance is 
the minimal number of indels (insertions/deletions) to transform one
string into the other and is therefore given by $2(n-LC_n)$.  (The reader will find in 
\cite{Cap}, \cite{Pevzner}, \cite{SK} and \cite{W2} 
numerous examples of the relevance of longest common subsequences in various
applications.) 

The asymptotic study of $LC_n$ began with the well known
result of Chv\'atal and Sankoff \cite{CS} asserting, via a superadditivity argument,  that 

\begin{equation}\label{lcmean}
\lim_{n \to \infty}\frac{\mathbb{E}LC_n}{n}= \gamma_m^*, 
\end{equation}
whenever, for example, $(X_i)_{i\ge 1}$ and $(Y_i)_{i\ge 1}$ are two independent
sequences of independent identically distributed (iid) random variables
having the same law.   

However, to this day, the exact value of $\gamma_m^*=\sup_{n\ge 1}\mathbb{E}LC_n/n$ 
(which depends on the distribution of $X_1$ and on the size of the alphabet) is 
unknown, even in "simple cases", such as for uniform Bernoulli
random variables.  Nevertheless, its asymptotic behavior, as the
alphabet size grows, is known (see Kiwi, Loebl and Matou$\breve{\rm s}$ek (\cite{KLM})) 
and given, for $X_1$ uniformly distributed, by:
\begin{equation}\label{gammaasymp}
\lim_{m \to \infty}\sqrt m\gamma_m^* = 2.
\end{equation}

Chv\'atal and Sankoff's law of large numbers was further sharpened by
Alexander (\cite{A}) who proved that
\begin{equation}\label{alex}
\gamma_m^*n -K_A\sqrt{n \ln n} \le \mathbb{E}LC_n \le \gamma_m^*n,
\end{equation}
where $K_A>0$ is a universal constant (which depends neither on $n$ nor
on the distribution of $X_1$).  Next, Steele \cite{St} obtained via the Efron--Stein inequality the first
upper bound on the variance of $LC_n$ proving, in particular, 
that: 
\begin{equation}\label{varup}
\Var LC_n\le \left(1-\sum_{k=1}^m p_k^2\right)n,
\end{equation}  
where $p_k = \mathbb{P}(X_1 = \alpha_k), k = 1, \dots, m$. 
However, finding the order of the lower bound is much more 
illusive and remains unknown in many instances, in particular for iid uniform Bernoulli 
random variables.   Some of the instances in which, and methods for which, a 
variance lower bound matching the linear upper bound have been obtained are further described below.  
Before doing so, let us state our main result:

\begin{thm}
\label{thm: CLT} Let $(X_i)_{i \geq 1}$ and $(Y_i)_{i \ge 1}$ be two
independent sequences of iid random variables with values in
$\cala_m=\{\aa_1,\aa_2, \dots, \aa_m\}$ and having the same law.  
Assume that 
$\Var LC_n \ge Kn$, for some positive constant $K$ independent of $n\ge 1$.   
Let $0< \eta < 1/10$, then for all $n\geq 1$, 
\begin{equation}\label{bound:Wasserstein}
d_W\left(\frac{LC_n - \mathbb{E}LC_n}{\sqrt{\Var LC_n}},
\mathcal{G}\right) \leq C\frac{1}{n^{\frac{1}{10}-\eta}},
\end{equation}
where  $d_W$ is the
Monge-Kantorovich-Wasserstein distance, where $\mathcal{G}$ a standard 
normal random variable and where $C>0$ is a constant depending on $K$, on $m$, and on 
the distribution of $X_1$, but is independent of
$n$.
\end{thm}

Recall next that the Kolmogorov and 
Monge-Kantorovich-Wasserstein distances, $d_K$ and $d_W$, between
two probability distributions $\mu_1$ and $\mu_2$ on $\mathbb{R}$,
are respectively defined as
$$d_K(\mu_1,\mu_2) = \sup_{h \in \mathcal{H}_1}
\left|\int h d\mu_1-\int h d\mu_2\right|,$$
where $\mathcal{H}_1=\{\mathbf{1}_{(-\infty, x]}: x\in
\mathbb{R}\}$, and
$$d_W(\mu_1,\mu_2)= \sup_{h \in \mathcal{H}_2}\left|\int h
d\mu_1  -\int h  d\mu_2 \right|,$$ where $\mathcal{H}_2=\{h :
\mathbb{R} \rightarrow \mathbb{R} : |h(x)-h(y)|\leq|x-y|\}$. Recall, further, that if
$\mu_2$ is absolutely continuous, with respect to the Lebesgue measure,  
and with density $\mu_2(dx)/dx$ essentially bounded, i.e., such that $\|\mu_2(dx)/dx\|_\infty < +\infty$,
then, 
\begin{equation}\label{KolmWassrelation}
	d_K(\mu_1,\mu_2) \leq \sqrt{2 \|\mu_2(dx)/dx\|_\infty d_W(\mu_1,\mu_2)},
\end{equation}
e.g., see Ross~\cite{ross} or the Appendix in \cite{AH}.  Thus, Theorem~\ref{thm:
	CLT} implies via \eqref{KolmWassrelation}, that 
\begin{equation}\label{bound:Kolmogorov}
	d_K\left(\frac{LC_n - \mathbb{E}LC_n}{\sqrt{\Var LC_n}},
	\mathcal{G}\right) \leq C^{1/2} \left(\frac{2}{\pi}\right)^{1/4}
	\frac{1}{n^{\frac{1}{20}-\frac{\eta}{2}}},
\end{equation}
and so, properly centered and normalized, $LC_n$ converges in distribution 
to a standard normal random variable as long as $\var LC_n$ is assumed to be of linear order.  

\

Let us carefully review and discuss the assumption on the variance of $LC_n$ present 
in the statement of our main theorem.  
As indicated in  \eqref{varup}, $\Var LC_n \le n$, however contradictory conjectures 
on the order of  this variance have also appeared in the literature: 
A sub-linear conjecture (of order $o(n^{2/3})$) 
in \cite{CS} and a linear one in Waterman~\cite{W} (see also \cite{A}).  
The linear order, which we believe to be the correct one, has been verified in a 
few situations that we briefly describe next:  

$\bf \bullet$ This linear lower bound is proved in \cite{LM} or \cite{HMa}  for iid random variables (Bernoulli or finite-alphabet ones) which 
are highly biased, in that a single letter is taken with very high (but fixed) probability.  In that case, changing in any configuration, a 
low probability letters into the high probability one, is more likely to increase $LC_n$ by one unit 
than to decrease it by one unit.
  This change (which clearly has
no effect for uniformly distributed letters) reduces variability and the new longest common subsequences provide the variance lower bound.  

$\bf \bullet$ Beyond the strongly biased 
cases just mentioned, a linear order for the variance 
has been obtained in other situations closer to the iid uniform case.  In particular, in a 
framework where either a letter is missing or long blocks 
are added within the iid uniform framework or in various other settings, as seen 
in the many references given in (\cite{AHM}, \cite{BM}, \cite{GHL}, \cite{HL}, \cite{HM}, \dots). .
Within these frameworks, modifications of the tools presented in our current  
approach would also lead to a 
central limit theorem, without any further assumption on the variance.   In all 
these situations, the central $r$-th, $r\ge1 $, moments of $LC_n$ can also 
be shown to be of order $n^{r/2}$ (see the concluding remarks in
\cite{HMa}). This last fact might hint at the asymptotic normality of
$LC_n$, although similar moments estimates can 
lead to a non-Gaussian limiting law in a related problem, i.e., in the study of  
$LCI_n$, the length of the longest common and increasing 
subsequences of two random words, over a totally ordered finite alphabet (see \cite{BH}, \cite{DH}).

$\bf \bullet$
Early extensive simulations (with $n$ of order $10^4$) 
by Boutet de Monvel~\cite{BdM} seemed to indicate, in the uniform case, 
a variance of order at least $n^{2\omega^\prime}$ with $\omega' \approx 0.418$ and 
even a normal  asymptotic law.  More recent extensive simulations (with $n$ of order 
$10^6$) (see \cite{LH}) seem to indicate  (in both the uniform and non-uniform binary cases) that 
the variance is of order $n$ as the lengths of the 
sequences studied there are the larger to date, an order one-hundred times bigger than the ones in \cite{BdM}.  

$\bf \bullet$ As it will become 
clear from the proof of the theorem just stated, 
a mere sublinear lower bound on the variance will also lead to a normal 
limiting law, e.g., a lower bound 
of order at least $n^{9/10 + \eta}$, $\eta > 0$ will do  
(although, and again, it is our belief that the variance of 
$LC_n$ is linear in $n$, but nevertheless $9/10  > 2\omega^\prime$).  
Note also that the proof of this theorem provides for $\alpha$ (to be defined) 
such that $4/5 < \alpha < 1$, a rate of 
$1/n^{{(1-\alpha)}/{2}}$, while for $2/3 < \alpha < 4/5$, a different rate,
of order $1/n^{1-3(1-\alpha/2)/2}$, 
can be obtained  in a similar way (see \eqref{Lastvarbound}), 
under a linear variance lower bound.

\begin{remark}\label{rmk:variance}
Theorem~\ref{thm: CLT} is the first of its kind.  It contrasts, in
particular, with the corresponding result in the related Bernoulli matching problem where, 
as shown by Majumdar and 
Nechaev (\cite{MN}), the
limiting law is the Tracy-Widom one.  Both the
LCS and Bernoulli matching models are directed last passage vertex/site percolation 
models with respectively dependent and independent weights, possibly
explaining the different limiting laws.  In both cases, the
expectation is linear in $n$, but the variance in the Bernoulli
matching problem is sublinear (of order $n^{2/3}$), while in our LCS
case it is assumed linear.   Let us describe how the LCS problem can be represented as a directed 
last passage percolation (LPP) problem with dependent weights.  
Indeed, let the set of vertices  be 
\[V:=\{0,1,2,\ldots,n\}\times \{0,1,2,\dots,n\},\] 
and let the set of oriented edges ${\cal E}\subset V\times V$ 
contain horizontal, vertical and diagonal edges.  
The horizontal edges are oriented to the right, 
while the vertical edges are oriented upwards, both having unit length. 
The diagonal edges point up-right at a $\pi/4$-angle and have length $\sqrt{2}$.  
Hence, 
\[{\cal E} :=\left\{ (v,v+e_1),(v,v+e_2),(v,v+e_3): v\in V\right\},\]
where $e_1:=(1,0)$, $e_2:=(0,1)$ and $e_3:=(1,1)$.  
With the horizontal and vertical edges, we associate 
a weight of $0$.  With the diagonal edge from $(i,j)$ to $(i+1,j+1)$ we associate 
the weight $1$ if $X_{i+1}=Y_{j+1}$ and $0$ (or $-\infty$) otherwise.  
In this manner, we obtain that $LC_n$  
is equal to the total weight of the heaviest paths going from $(0,0)$ to $(n,n)$.  
(Another directed LPP representation can be obtained 
via $LC_n = \max_{\pi \in SI}\sum_{(i,j)\in \pi}{\bf 1}_{\{X_i = Y_j\}}$, 
where $SI$ refers to the set of all paths with \textit{strictly} increasing steps, i.e., paths 
with \textit{both} coordinates strictly increasing from a step to another, from $(0,0)$ to 
the East, $x=n$, or North, $y=n$, boundary.  A third representation would be as above 
but where now the paths 
going from $(0,0)$ to $(n,n)$ have either strictly increasing steps 
or North or East unit steps.  
Again to the strictly 
increasing steps the associated weight is ${\bf 1}_{\{X_i = Y_j\}}$ 
while to the North as well as to the East unit steps is associated a weight value of $0$.  
As a final representation one could still proceed with strictly increasing paths but 
with the requirement that one ends the paths with a $1$.)   
Note that the weights in our percolation representations are not ``truly 2-dimensional'' 
and, in our opinion, this could be a further 
reason for the order of magnitude of the mean, variance as well 
as the limiting law in the LCS problem to be different from 
other first/last passage-related models. 

Theorem~\ref{thm: CLT} further contrasts with the
corresponding limiting law for the length of the longest common 
subsequences in a pair of independent uniform random permutations of
$\{1, \dots, n\}$.  Indeed, in sequence comparison problems, the emergence of the
Tracy--Widom distribution has sometimes been contemplated/speculated, e.g., see \cite{AD}. We show, in the last section of the present paper, that this is correct when analyzing the asymptotic behavior of the length of the longest common subsequences of two 
independent uniform random permutations of $\{1, \dots, n\}$ (the expectation there is of order $\sqrt n$ and the
variance of order $n^{1/3}$).  

Finally, let us remark that some of the ideas/techniques developed to prove lower bounds on 
$Var LC_n$ have been further developed in the context of first passage percolation, providing, to date, the best lower bound available  on the variance of the passage time (see \cite{DHHX}).

\end{remark}

As far as the content of the paper is concerned, the lengthy next section
contains the proof of Theorem~\ref{thm: CLT}, which is preceded 
by a discussion of  some elements of its proof.  Then, in the third section,
various extensions and generalizations as well as some related open
questions are discussed.  In particular, the proof, that the length
of longest common subsequences in two independent uniform random permutations
of $\{1, \dots, n\}$ converges to the Tracy-Widom distribution, is included there.

\section{Proof of Theorem~\ref{thm: CLT}}

The aim of this section is to provide a proof of the main theorem by 
a three-step method.  The first step makes use of a relatively recent theorem of 
Chatterjee (\cite{chat1}) on Stein's method (see \cite{cgs} for an
overview of the method, including Chatterjee's normal approximation
results via exchangeable pairs); the second uses simple moment 
estimates for $LC_n$ derived from our 
lower bound variance assumption; and 
the third develops lengthy correlation estimates based, in part, on short 
string-lengths genericity results obtained in \cite{HMdiag}. We  
start by fixing notation and recalling some preliminaries.

Throughout this section,  $X=(X_i)_{i\ge 1}$ and $Y=(Y_i)_{i\ge 1}$
are two independent sequences whose coordinates are iid, with a common law,  
and taking their values in $\cala_m=\{\aa_1,\aa_2, \dots, \aa_m\}$, a 
finite alphabet of size $m$.  

Let us continue by introducing some more notation following those of
\cite{chat1}. Let $W=(W_1,W_2,\dots,W_n)$ and
$W'=(W_1',W_2',...,W_n')$ be two iid 
$\mathbb{R}^n$-valued random vectors whose components are also 
independent.  For $A \subset [n]:=\{1,2,\dots,n\}$, define the 
random vector $W^A$ by setting 
\[ W_i^A = \begin{cases}
      W_i' & \textrm{ if $i \in A$} \\
      W_i & \textrm{ if $i \notin A$}, \\
   \end{cases} \]
with for $A=\{j\}$, and further ease of notation, $W^j$ is short for $W^{\{j\}}$, while $W^\emptyset = W$.  

For a given Borel measurable function $f : \mathbb{R}^n \rightarrow
\mathbb{R}$ and $A \subset [n]$, let
$$T_A:=\sum_{j \notin A} \Delta_j f(W) \Delta_j f(W^A),$$ where
\begin{equation*}
    \Delta_jf(W):= f(W)-f(W^j), 
\end{equation*}
and again, $T_\emptyset =\sum_{j =1}^n (\Delta_j f(W) )^2$.  Finally, let
\begin{equation*}
    T =\frac{1}{2} \sum_{A \subsetneq [n]}  \frac{T_A}{\binom{n}{|A|}
    (n-|A|)},
\end{equation*}
where $|A|$ denotes the cardinality of $A$, and where the sum above is taken over 
all the proper subsets (including $T_\emptyset$) of $[n]$.   
Here is Chatterjee's result.

\begin{theorem}\label{chat1} \cite{chat1}
Let all the terms be defined as above, and let $0<\sigma^2:= \Var
f(W) < \infty$. Then,
\begin{equation}\label{chatbound}
    d_W\left(\frac{f(W) - \mathbb{E}f(W)}{\sqrt{\Var f(W)}},\mathcal{G}\right)
    \le \frac{\sqrt{\Var T}}{\sigma^2} +
    \frac{1}{2\sigma^3}\sum_{j=1}^n \mathbb{E}|\Delta_jf(W)|^3,
\end{equation}
where $\mathcal{G}$ is a standard normal random variable.
\end{theorem}

\begin{remark} (i) In \cite{chat1}, the variance term as displayed in \eqref{chatbound} 
is actually replaced by $\Var \mathbb{E}(T|f(W))$ but the above bound, with the
larger $\Var T$, already presented in \cite{chat1}, is enough for our
purpose.

(ii)  Our proof bounds the right-hand side of \eqref{chatbound} and next, using  
\eqref{KolmWassrelation}, bounds the corresponding Kolmogorov distance.   
An alternate way to obtain convergence 
in distribution would be to first use a  more recent result 
of Lachi\`eze-Rey and Peccati \cite{LRP}, directly bounding the Kolmogorov distance, which could then 
be estimated by adapting the techniques 
presented below.

\end{remark}

Two small comments are in order before beginning the proof of Theorem~\ref{thm: CLT}.

\begin{itemize}
  \item[(1)]
In the proof, we do not keep  track of the constants since doing so 
would make the arguments a lot lengthier.  Therefore, a constant $C$
may vary from an expression to another.  Note, however, that $C$ will 
always be positive and independent of $n$. 
\item[(2)] We do not worry about having quantities (e.g. length of longest 
  common subsequences of two random words) like $n^{\alpha}, \ln n, etc.$ which 
  should actually be $\lfloor n^{\alpha}\rfloor$, $\lfloor \ln
n\rfloor$, etc. This does not cause any problems as we are interested in
asymptotic bounds. The proof can be revised with minor changes (and
some further notational burden) to make the statements more precise.  
\end{itemize}
Let us start with a sketch of proof Theorem~\ref{thm: CLT} and to do so, set 

\begin{equation}\label{Wdefinition}
    W:=(X_1,\dots, X_n,Y_1,\dots, Y_n),
\end{equation}
and set 
$$f(W):=LC_n(X_1\cdots X_n;Y_1\cdots Y_n).$$
We begin by estimating the second term on the right-hand side of
\eqref{chatbound}.  To do so, recall our assumption:  

\begin{equation}\label{HMavarestimate}
\sigma^2: = \mathbb{E}(LC_n- \mathbb{E}LC_n)^2 \geq K n. 
\end{equation}

\noindent
Therefore,
\begin{equation}\label{momentlowerbound}
    \sigma^3 \geq C n^{3/2}, \qquad n\geq 1,
\end{equation}
yielding
\begin{equation}\label{est:2ndterm}
\frac{1}{2\sigma^3}\sum_{j=1}^{2n} \mathbb{E}|\Delta_jf(W)|^3 \leq C
\frac{1}{\sqrt{n}},
\end{equation}
since $|\Delta_jf(W)| \leq 1$.   This last estimate takes care of the second term on the right-hand 
side of \eqref{chatbound}.  

Next, let us move to the estimation of the variance term in 
\eqref{chatbound}. Setting \begin{equation}\label{def:S}
    \mathcal{S}_1 := \{(A,B,j,k) : A \subsetneq [2n], B \subsetneq [2n],j \notin A, k \notin
    B\},
\end{equation} $\Var T$ can be expressed as \begin{eqnarray}\label{varbasicexpansion}
\nonumber  \Var T  &=& \frac{1}{4} \Var \left(\sum_{A \subsetneq
[2n]} \sum_{j \notin A}
\frac{\Delta_j f(W) \Delta_j f(W^A)}{\binom{2n}{|A|} (2n-|A|)} \right) \\
\nonumber    &=& \frac{1}{4}\! \sum_{A \subsetneq [2n],j \notin A}
\sum_{B \subsetneq [2n],k \notin B}
    \frac{Cov(\Delta_j f(W) \Delta_jf(W^A), \Delta_k f(W) \Delta_kf(W^B))}{\binom{2n}{|A|}
    (2n-|A|)\binom{2n}{|B|} (2n-|B|)}\\
   &=& \frac{1}{4} \sum_{(A,B,j,k) \in \mathcal{S}_1} \frac{Cov(\Delta_j f(W) \Delta_jf(W^A),
   \Delta_k f(W) \Delta_kf(W^B))}{\binom{2n}{|A|} (2n-|A|)\binom{2n}{|B|}
   (2n-|B|)}.
\end{eqnarray}

Our strategy is now to further divide $\mathcal{S}_1$ into two main pieces  by conditioning on a, yet to be defined, high probability 
event $E^n_{\epsilon, s_1, s_2}$, ensuring that LCSs are made of an accumulation of relatively short strings.  
More precisely, 
\begin{lemma}\label{decompS1} Let $Z= {\bf 1}_{E^n_{\epsilon, s_1,s_2}}$, then 
\begin{eqnarray}	
4\Var T  
	&=&\!\!\!\!\!\! \!\!\!\!\!\!\!\sum_{(A,B,j,k) \in \mathcal{S}_1} \!\!\!\!\!\!\!\!\frac{Cov(\Delta_j f(W) \Delta_jf(W^A),
		\Delta_k f(W) \Delta_kf(W^B)\!)}{\binom{2n}{|A|} (2n-|A|)\binom{2n}{|B|}
		(2n-|B|)} \mathbb{P}(\!Z=0)\\
	&+&\!\!\!\!\!\!\!\!\!\!\!\!\!\!\sum_{(A,B,j,k) \in \mathcal{S}_1} \!\!\!\!\!\!\!\!\frac{Cov(\Delta_j f(W) \Delta_jf(W^A),
		\Delta_k f(W) \Delta_kf(W^B)\!)}{\binom{2n}{|A|} (2n-|A|)\binom{2n}{|B|}
		(2n-|B|)} \mathbb{P}(\!Z=1)\!. 
\end{eqnarray} 
\end{lemma}

To estimate each of the two terms in the above lemma, the following proposition, and a conditional version of it, which
easily follows from similar arguments, will be used repeatedly
throughout the proof.
\begin{proposition}\label{prop:combcomp}
Let $\mathcal{R}$ be a subset of $[2n]^2$, and let
$$\mathcal{S}^* = \{(A,B,j,k) : A \subsetneq [2n], B \subsetneq
[2n],j \notin A, k \notin
    B, (j,k) \in \mathcal{R}\}.$$  Let $g : \mathcal{S}^* \rightarrow \mathbb{R}$
    with $\|g\|_{\infty} < +\infty$,
    then
$$\sum_{(A,B,j,k)\in \mathcal{S}^*}  \left|\frac{g(A,B,j,k)}{\binom{2n}{|A|}
(2n-|A|) \binom{2n}{|B|} (2n-|B|)}\right| \leq \|g\|_{\infty} |\mathcal{R}|.$$
\end{proposition}
\begin{proof}
First,  since $\|g\|_{\infty} < +\infty$,
\begin{eqnarray*}
\lefteqn{\sum_{(A,B,j,k)\in \mathcal{S}^*}  \left|
\frac{g(A,B,j,k)}{\binom{2n}{|A|} (2n-|A|)\binom{2n}{|B|}
(2n-|B|)}\right|}  \\& \leq&
 \|g\|_\infty \sum_{(A,B,j,k)\in \mathcal{S}^*}
   \left( \frac{1}{\binom{2n}{|A|} (2n-|A|)\binom{2n}{|B|} (2n-|B|)}\right).
\end{eqnarray*}
Next, expressing $\sum_{(A,B,j,k)\in \mathcal{S}^*}$ in terms of
$\mathcal{R}$, using basic results about binomial coefficients and
performing some elementary manipulations lead to
\begin{eqnarray*}
\lefteqn{\sum_{(A,B,j,k)\in \mathcal{S}^*} \frac{1}{\binom{2n}{|A|}
(2n-|A|)\binom{2n}{|B|}
(2n-|B|)}} \\
&=& \sum_{(j,k)\in \mathcal{R}} \sum_{\substack{A \subsetneq [2n] :
A \not \ni j  \\ B \subsetneq [2n] : B \not \ni k }}
\frac{1}{\binom{2n}{|A|}
(2n-|A|)\binom{2n}{|B|} (2n-|B|)} \\
&=& \sum_{(j,k)\in \mathcal{R}} \left(\sum_{s,r=0}^{2n-1}
\sum_{\substack{A \not \ni j, |A|=s\\ B \not \ni k, |B|=r}}
\frac{1}{\binom{2n}{|A|} (2n-|A|)\binom{2n}{|B|} (2n-|B|)}\right) \\
&=& \sum_{(j,k)\in \mathcal{R}} \left(\sum_{s,r=0}^{2n-1}
\sum_{\substack{A \not \ni j, |A|=s\\ B \not \ni k, |B|=r}}
\frac{1}{\binom{2n}{s} (2n-s)\binom{2n}{r} (2n-r)}\right) \\
&=& \sum_{(j,k)\in \mathcal{R}} \left(\sum_{s,r=0}^{2n-1}
\frac{\binom{2n-1}{s}\binom{2n-1}{r}}{\binom{2n}{s}
(2n-s)\binom{2n}{r} (2n-r)}\right) \\
&=&\sum_{(j,k)\in \mathcal{R}} \left(\sum_{s,r=0}^{2n-1}
\frac{\frac{(2n-1)!}{(2n-1-s)!s!}\frac{(2n-1)!}{(2n-1-r)!r!}}{\frac{(2n)!}{(2n-s)!s!}
	(2n-s)\frac{(2n)!}{(2n-r)!r!} (2n-r)}\right) \\
&=&  \sum_{(j,k)\in \mathcal{R}}
\left(\sum_{s,r=0}^{2n-1} \frac{1}{(2n)^2} \right) \\
&=& |\mathcal{R}|,
\end{eqnarray*}
from which the result follows.
\end{proof}
Taking $\mathcal{R}=[2n]^2$, $g(A,B,j,k) = Cov(\Delta_j f(W)
\Delta_jf(W^A), \Delta_k f(W) \Delta_kf(W^B))$ which is such that $\|g\|_{\infty} \le 1$, Proposition~\ref{prop:combcomp} 
yields the estimate
\begin{equation}\label{MasterEq}
\sum_{(A,B,j,k)\in \mathcal{S}_1} \left(\frac{Cov(\Delta_j f(W)
\Delta_jf(W^A), \Delta_k f(W) \Delta_kf(W^B))}{\binom{2n}{|A|}
(2n-|A|)\binom{2n}{|B|} (2n-|B|)}\right) \leq 4n^2.
\end{equation}
Hence, $\Var T \le n^2$ giving a suboptimal result 
for our purposes, and we therefore begin a detailed 
estimation study to improve the variance upper bound to $o(n^2)$.

To do so, we start by giving a slight variation of a result from
\cite{HMdiag} which can be viewed as a microscopic short-lengths 
genericity principle, and which will turn out to be an important
tool in our proof.  This principle, valid not only for common
sequences but in much greater generality (see \cite{HMdiag}), should
prove useful in other contexts.

Assume that $n=vd$, and let the integers
\begin{equation}\label{optcond1}
    r_0=0\leq r_1 \leq r_2\leq r_3\leq ...\leq r_{d-1}\leq r_d=n,
\end{equation}
be such that
\begin{equation}\label{optcond2}
LC_n=\sum_{i=1}^d |LCS(X_{v(i-1)+1}X_{v(i-1)+2} \cdots
X_{vi};Y_{r_{i-1}+1}Y_{r_{i-1}+2} \cdots Y_{r_i})|,
\end{equation}
where $|LCS(X_{v(i-1)+1}X_{v(i-1)+2} \cdots
X_{vi};Y_{r_{i-1}+1}Y_{r_{i-1}+2} \cdots Y_{r_i})|$ is the length of
the longest common subsequences of the words/strings 
$X_{v(i-1)+1}X_{v(i-1)+2} \cdots X_{vi}$ and
$Y_{r_{i-1}+1}Y_{r_{i-1}+2} \cdots Y_{r_i}$ (with the understanding
that this length is zero if none of the letters of the $X$-part are aligned with letters of the 
$Y$-part, i.e., if the $X$-part is only aligned with gaps).
Next, let $\epsilon >0$ and let $0<s_1<1<s_2$, be two reals such
that
$$\tilde{\gamma}(s_1)<\tilde{\gamma}(1)=\gamma_m^* \quad \text{and}
\quad   \tilde{\gamma}(s_2)<\tilde{\gamma}(1)=\gamma_m^*,$$ where
$$\tilde{\gamma}(s)= \lim_{n \rightarrow \infty}
\frac{\mathbb{E}LC_n(X_1\cdots X_n;Y_1 \cdots  Y_{sn})}{n(1+s)/2},
\quad s>0.$$
(See \cite{HMdiag} for the existence of, and estimates on, $s_1$ and
$s_2$.) 

Finally, let $E_{\epsilon,s_1,s_2}^n$ be the event that for
all integer vectors $(r_0,r_1,...,r_d)$ satisfying \eqref{optcond1}
and \eqref{optcond2}, we have
\begin{equation}\label{optcond3}
    |\{i \in [d] : vs_1 \leq r_i - r_{i-1} \leq v s_2\}| \geq
    (1-\epsilon) d.
\end{equation}

In words, $E^n_{\epsilon,p_1,p_2}$ 
is the (random) set of optimal alignments 
of $X_1\cdots X_n$ and $Y_1\cdots Y_n$ satisfying \eqref{optcond1}, 
for which a proportion of at least $1-\epsilon$ 
of the integer intervals $[r_{i-1}+1,r_{i}]_{\bbn}$, $i=1,2,\dots,d$, 
have their length between $vs_1$ and $vs_2$.  

As stated next, $E^n_{\epsilon,s_1,s_2}$ holds with high probability.  
Broadly, our next theorem asserts that for any $\epsilon > 0$, 
there exists 
$v$ large enough, but fixed, such that if $X$ is divided into segments of length $v$ then,  
typically (at least a fraction $1-\epsilon$ of segments), and with high probability, 
the LCSs match these segments to segments of similar length in $Y$.

\begin{theorem}\label{thm: HMdiag} \cite{HMdiag}
Let $\epsilon >0$.  Let $0<s_1<1<s_2$ be such that
$\tilde{\gamma}(s_1)<\tilde{\gamma}(1)=\gamma_m^*$ and
$\tilde{\gamma}(s_2)<\tilde{\gamma}(1)=\gamma_m^*$, and let $\delta
\in (0, \min(\gamma_m^*-\tilde{\gamma}(s_1),
\gamma_m^*-\tilde{\gamma}(s_2)))$.  Let the integer $v$ be such that 
\begin{equation}\label{HMdiagcond}
\frac{1+\ln\,(1+v)}{v} \le \frac{\delta^2 \epsilon^2}{16}.
\end{equation}
Then,
\begin{equation}\label{HMdiagbound}
\mathbb{P}(E_{\epsilon,s_1,s_2}^n) \geq 1 -
\exp\left(-n\left(-\frac{1 +\ln\,(1+v)}{v} +\frac{\delta^2
\epsilon^2}{16} \right) \right),
\end{equation}
for all $n=n(\epsilon, \delta)$ large enough.
\end{theorem}

\begin{remark}\label{rmk:existence}
In \cite{HMdiag},  instead of \eqref{optcond1}, the corresponding condition is:
\begin{equation}\label{optcondHM}
    r_0=0< r_1 < r_2< r_3< ...< r_{d-1}< r_d=n.
\end{equation}
which is made up of strict inequalities becoming weak inequalities in  \eqref{optcond1}.  
The rationale for this difference is that, in general, there is no guarantee that there exists an
optimal alignment, i.e., a longest common subsequence, satisfying both 
conditions \eqref{optcond2} and
\eqref{optcondHM}.  Indeed, for a simple counterexample, let $n=4$,
$\mathcal{A}=[2]$, $d=v=2$, and let $$X=(1,1,0,0), \qquad
Y=(0,0,1,1).$$  Then, any optimal alignment satisfying
\eqref{optcond2} must have a piece (soon to be called a ''cell") with no terms 
in the $Y$-part and this is clearly incompatible with \eqref{optcondHM}. (This
counterexample can easily be extended to $n=6$, $\mathcal{A}=[2]$,
$d=3, v=2$, letting $X=(1,1,0,0,1,1)$, $Y=(0,0,1,1,0,0)$, and so
on.)

In general, there always exists an optimal alignment 
$(r_0,r_1,r_2,...,r_d)$ satisfying both \eqref{optcond1} and
\eqref{optcond2} with, say, $v=n^{\alpha}$, $0 < \alpha < 1$, as above. 
(Consider any
one of the longest common subsequences and choose the $r_i$'s so
that these two conditions are satisfied.) Therefore, we slightly
change the framework of \cite{HMdiag} as forthcoming arguments 
require the existence of an optimal alignment satisfying \eqref{optcond2}
for any value of $X$ and $Y$.  However, the proof of
Theorem~\ref{thm: HMdiag}, above, proceeds as the proof of the corresponding
result (Theorem 2.2!) in \cite{HMdiag}, and is therefore omitted. (The only
difference is that counting the cases of equality, an upper estimate
on the number of integer-vectors $(0=r_0, r_1, \dots, r_{d-1}, r_d
=n)$ satisfying \eqref{optcond1} is now given by
\begin{equation}\label{newbound}
 \binom {n+d}{d} \le \frac{(n+d)^d}{d!}\le \left(\frac{e(n+d)}{d}\right)^d 
 = (e(1+v))^d,
\end{equation}
leading to the terms involving $\ln\,(1+v)$ rather than just
$\ln\,v$ \cite{HMdiag}, when using \eqref{optcondHM} and the estimate $n^d/d! \le (ev)^d$ to upper-bound
$\binom{n}{d}$.)
\end{remark}

\begin{remark}\label{rmk:HMdiag} In \cite{HMdiag}, the statement corresponding to 
Theorem~\ref{thm: HMdiag} is given for "all $n$ large enough".  
However, as indicated at the end of the proof there, it is possible
to find a more quantitative estimate using Alexander's result 
\eqref{alex}. In fact a lower bound, in terms of
$\epsilon$ and $\delta$, is valid for all $n\ge 1$. Indeed, at first, from the
end of the proof of Lemma 3.1 there, preceding the main theorem in \cite{HMdiag}, one can
easily verify that the following relation between $n$ and $\epsilon$ is sufficient for 
\eqref{HMdiagbound} to hold:  
$$\frac{4K_A^2}{(\delta^*-\delta)^2} \frac{\ln n}{n} \le \epsilon^2,$$
where $0 < \delta < \delta^*:= \min(\gamma_m^*-\tilde{\gamma}(s_1),
\gamma_m^*-\tilde{\gamma}(s_2))$ is a fixed positive quantity and $K_A$ is a
positive constant such that $\gamma_m^* n - K_A\sqrt{n \ln n} \leq
\mathbb{E}LC_n.$ (One can find explicit numerical estimates on $K_A$
using Rhee's \cite{rhee} proof of \eqref{alex}.)

In our context, here is how to choose $\epsilon$ so that the
estimate in \eqref{HMdiagbound} holds true for all $n \geq 1$ and
$v=n^{\alpha}$, $0 < \alpha <1$. Let $c_1>0$ be a constant such that
$$c_1^2\geq \frac{32}{\delta^2},$$ and
$$c_1^2 \left(\frac{1+ \ln\,(1 + n^\alpha)}{n^{\alpha}} \right) \geq \frac{4K_A^2}
{(\delta^*-\delta)^2} \frac{\ln n}{n}, \qquad \text{for all} \; n
\geq 1.$$ Setting, $$\epsilon^2 =
c_1^2\frac{1+\ln\,(1+n^{\alpha})}{n^{\alpha}},$$ \eqref{HMdiagcond}
holds for $v=n^{\alpha}$ and therefore,
\begin{equation}\label{Zequals0probest}
\mathbb{P}(E_{\epsilon,s_1,s_2}^n) \ge 1 - e^{-n^{1-\alpha
}(1+\ln\,(1 + n^\alpha))}\ge 1-e^{-(1+\ln2)},
\end{equation}
for all $n\ge 1$.
\end{remark}

Let us return to the proof of Theorem~\ref{thm: CLT}, and the
estimation of \eqref{varbasicexpansion}.  First, for notational
convenience, below we write $\sum_{\mathcal{S}_1}$ in place of $\Sigma_{(A,B,j,k)
\in \mathcal{S}_1}$. Also, for random variables $U, V$ and a random
variable $Z$ taking its values in $R \subset \mathbb{R}$, and with
another abuse of notation, we write $Cov_{Z=z} (U,V)$ for
$\mathbb{E}((U-\mathbb{E}U)(V-\mathbb{E}V)|Z=z)$, $z\in R$.

Let, now, the random variable $Z$ be the indicator function of the
event $E_{\epsilon,s_1,s_2}^n$, where $\epsilon = c_1\sqrt{(1 +
\ln\,(1+v)) / v}$, i.e., let $Z={\bf 1}_{E_{\epsilon,s_1,s_2}^n}$, with
$v = n^{\alpha}$ and with $c_1$ as in Remark~\ref{rmk:HMdiag}. 
Then, we arrive at the decomposition of 
Lemma~\ref{decompS1}
\begin{eqnarray}\label{firstcov}
  \sum_{\mathcal{S}_1}\lefteqn{\frac{Cov(\Delta_j f(W) \Delta_jf(W^A),
  \Delta_k f(W) \Delta_kf(W^B))}{\binom{2n}{|A|} (2n-|A|)\binom{2n}{|B|} (2n-|B|)}}
  \nonumber\\
  &=&\!\!\!\!\!\sum_{\mathcal{S}_1} \frac{Cov_{Z=0}(\Delta_j f(W) \Delta_jf(W^A),
  \Delta_k f(W) \Delta_kf(W^B))}{\binom{2n}{|A|} (2n-|A|)\binom{2n}{|B|} 
  (2n-|B|)} \mathbb{P}(Z=0)
  \nonumber\\
  &+& \!\!\!\!\!\sum_{\mathcal{S}_1}\! \frac{Cov_{Z=1}(\Delta_j f(W) \Delta_jf(W^A\!), \Delta_k f(W\!)
  \Delta_kf(W^B)\!)}{\binom{2n}{|A|} (2n-|A|)\binom{2n}{|B|} (2n-|B|)} \mathbb{P}(\!Z=1\!).
\end{eqnarray}
To estimate the first term on the right-hand side of
\eqref{firstcov}, first note that $Cov_{Z=0}(\Delta_j f(W)
\Delta_jf(W^A), \Delta_k f(W) \Delta_kf(W^B)) \le 1$, which when
combined with the estimate in \eqref{Zequals0probest} and
\eqref{MasterEq}, immediately leads to
\begin{eqnarray}\label{est:sum20}
\nonumber \sum_{\mathcal{S}_1} \frac{Cov_{Z=0}(\Delta_j f(W) \Delta_jf(W^A),
\Delta_k f(W) \Delta_kf(W^B))}{\binom{2n}{|A|}
(2n-|A|)\binom{2n}{|B|} (2n-|B|)}
    \mathbb{P}(Z=0) && \\ \quad \quad \quad \quad \le  4n^2 e^{-n^{1-\alpha
}(1+\ln\,(1+n^\alpha))}.
\end{eqnarray}
For the second term on the right-hand side of \eqref{firstcov},
begin with the trivial bound on $\mathbb{P}(Z=1)$ to get
\begin{eqnarray}\label{est:sum1}
\sum_{\mathcal{S}_1} \frac{Cov_{Z=1}(\Delta_j f(W) \Delta_jf(W^A),
\Delta_k f(W) \Delta_kf(W^B))}{\binom{2n}{|A|}
(2n-|A|)\binom{2n}{|B|} (2n-|B|)}
   \mathbb{P}(Z=1) && \nonumber \\
   \leq \sum_{\mathcal{S}_1}
\frac{Cov_{Z=1}(\Delta_j f(W) \Delta_jf(W^A), \Delta_k f(W)
\Delta_kf(W^B))}{\binom{2n}{|A|} (2n-|A|)\binom{2n}{|B|} (2n-|B|)}.
&&
\end{eqnarray}

Finer decompositions are then needed to handle this last summation,
and for this purpose, we specify an optimal alignment with certain
properties.

Recall from Remark~\ref{rmk:existence} that there always exists an
optimal alignment $\mathbf{r}=(r_0,r_1,r_2,...,r_d)$ satisfying both
\eqref{optcond1} and \eqref{optcond2} with $v=n^{\alpha}$, $0 < \alpha < 1$.
  In the sequel, $\mathbf{r}$ denotes such a (fixed) optimal
alignment which also specifies the pairs, in the strings $X_1\cdots X_n$ and
$Y_1\cdots Y_n$, contributing to the longest common subsequence.\footnote{This alignment might change according to different 
realizations of the words $X_1\cdots X_n$ and $Y_1\cdots Y_n$, i.e., it is random.  But for each realization, and although there are possibly more than 
one choice of optimal alignments, we just fix one of those.} Such an
alignment always exists, as just noted, and so we can define an
injective map from $(X_1\cdots X_n,Y_1\cdots Y_n)$ to the set of alignments, making 
various definitions (such as the ones for $\mathcal{S}_{1,1}$ and $\mathcal{S}_{1,2}$, below) 
well defined. 
This abstract
construction is enough for our purposes, since the argument below is
independent of the choice of the alignment.  Note also that
conditionally on the event $\{Z=1\}$, $\mathbf{r}$ satisfies
\eqref{optcond3}.

To continue, we need another definition and some more notation.

\begin{definition}
For the optimal alignment $\mathbf{r}$, each of the sets
$$
\{X_{v(i-1)+1}X_{v(i-1)+2} \cdots
 X_{vi};Y_{r_{i-1}+1}Y_{r_{i-1}+2} \cdots Y_{r_i}\}, \qquad i=1,...,d,
$$
is called a \emph{cell} of $\mathbf{r}$.
\end{definition}

In particular, any optimal alignment with
$v=n^{\alpha}$ has $d=n^{1-\alpha}$ cells.

Let us next introduce some more notation which will be used below. For any given
$j\in [2n]$, let $P_j$ be the cell containing $W_j$ where, again,
$W=(W_1, \dots, W_{2n})=(X_1, \dots, X_n,Y_1, \dots, Y_n)$.  We
write $P_j=(P_j^1;P_j^2)$ where $P_j^1$ (resp.~$P_j^2$) is the
subword of $X$ (resp.~$Y$) corresponding to $P_j$. Note that, for
each $j \in [2n]$, $P_j^1$ contains $n^{\alpha}$ letters but that
$P_j^2$ might be empty, as the following example shows:

\begin{example}
Let $n=12$ and $\mathcal{A}= [3]$, and let 
$$X=(1,1,2,1,2,1,1,2,1,1,3,1),$$
$$Y=(2,1,1,3,2,3,1,2,1,1,1,2).$$
and $W=(X,Y)$.  Then, $LC_{12}=8$, obtained for example through $(1,1,2,1,2,1,1,1)$, while choosing $v=3$, the number of
cells in the optimal alignment is $d=4$. One possible choice for
these cells is
$$(X_1 X_2 X_3; Y_1 Y_2 Y_3 Y_4 Y_5) =(1 1 2;2 1 1 3 2),$$
$$(X_4 X_5 X_6; \emptyset) =(1 2 1; \emptyset),$$
$$(X_7 X_8 X_9; Y_6 Y_7 Y_8 Y_9) =(1 2 1;3 1 2 1),$$ and
$$(X_{10} X_{11} X_{12}; Y_{10} Y_{11} Y_{12}) =(1 3 1;1 1 2).$$
For example, focusing on $W_8=X_8$, we have $$P_8 = (P_8^1;P_8^2)
=(1 2 1;3 1 2 1).$$ 
\end{example}

Returning to the proof of Theorem~\ref{thm: CLT}, define the
following subsets of $\mathcal{S}_1$ with respect to the alignment
$\mathbf{r}$:
$$\mathcal{S}_{1,1} = \{(A,B,j,k) \in \mathcal{S}_1 : W_j \; \text{and} \; W_k \;
\text{are in the same cell of $\mathbf{r}$}\},$$ and
$$\mathcal{S}_{1,2} = \{(A,B,j,k) \in \mathcal{S}_1 : W_j \; \text{and} \; W_k \;
\text{are in different cells of $\mathbf{r}$}\}.$$ \label{s11s12}
Clearly,
$\mathcal{S}_{1,1}\cap \mathcal{S}_{1,2} = \emptyset$ and
$\mathcal{S}_1 = \mathcal{S}_{1,1} \cup \mathcal{S}_{1,2}$. Now, for
a given subset $\mathcal{S}$ of $\mathcal{S}_1$, and for
$(A,B,j,k)\in \mathcal{S}_1$, define
$Cov_{Z=1,(A,B,j,k),\mathcal{S}}$ to be
$$Cov_{Z=1,(A,B,j,k),\mathcal{S}}(X,Y) = \mathbb{E}\left((X-\mathbb{E}X)
(Y-\mathbb{E}Y) \mathbf{1}_{(A,B,j,k) \in \mathcal{S}} | Z=1\right),$$
and, moreover, write $Cov_{Z=1,\mathcal{S}}(X,Y)$ instead of
$Cov_{Z=1,(A,B,j,k),\mathcal{S}}(X,Y)$ when the value of $(A,B,j,k)$
is clear from the context.

Continuing with the decomposition of the right-hand side of
\eqref{est:sum1}, 
\begin{eqnarray}\label{est:S22}
\nonumber \lefteqn{\sum_{\mathcal{S}_1} \frac{Cov_{Z=1}(\Delta_j f(W)
\Delta_jf(W^A), \Delta_k f(W) \Delta_kf(W^B))}{\binom{2n}{|A|}
(2n-|A|)\binom{2n}{|B|} (2n-|B|)}  }\\ &=&  \nonumber  \sum_{\mathcal{S}_1}
\frac{Cov_{Z=1, \mathcal{S}_{1,1}}(\Delta_j f(W) \Delta_jf(W^A),
\Delta_k f(W) \Delta_kf(W^B))}{\binom{2n}{|A|}
(2n-|A|)\binom{2n}{|B|} (2n-|B|)} \\
 &&\!\!\quad + \sum_{\mathcal{S}_1} \frac{Cov_{Z=1, \mathcal{S}_{1,2}}(\Delta_j f(W) \Delta_jf(W^A),
\Delta_k f(W) \Delta_kf(W^B))}{\binom{2n}{|A|}
(2n-|A|)\binom{2n}{|B|} (2n-|B|)}, 
\end{eqnarray}
where to further clarify the notation note that, for example,
\begin{eqnarray}
\nonumber \lefteqn{\sum_{\mathcal{S}_1} \frac{Cov_{Z=1,
\mathcal{S}_{1,1}}(\Delta_j f(W) \Delta_jf(W^A), \Delta_k f(W)
\Delta_kf(W^B))}{\binom{2n}{|A|}
(2n-|A|)\binom{2n}{|B|} (2n-|B|)}}\\
&& \qquad = \nonumber  \sum_{\mathcal{S}_1} \mathbb{E} \left(\frac{ g(A,B,j,k)
\mathbf{1}_{(A,B,j,k) \in \mathcal{S}_{1,1}}}{\binom{2n}{|A|}
(2n-|A|)\binom{2n}{|B|} (2n-|B|)} \Big| Z=1\right),
\end{eqnarray}
where
\begin{eqnarray}\label{defn:g}
g(A,B,j,k)\!\!\!\!&=&\!\!\!\!\left(\Delta_jf(W) \Delta_jf(W^A) - \mathbb{E}(\Delta_j
f(W)\Delta_jf(W^A))\right) \nonumber \\
 &\quad&  \!\times \!\left(\Delta_k
f(W) \Delta_kf(W^B)-\mathbb{E}(\Delta_k f(W) \Delta_kf(W^B))\right).
\end{eqnarray}

To glimpse into the proof, let us stop for a moment to present some
of its key steps. Our first intention is to show that, thanks to our
conditioning on the event $E_{\epsilon,s_1,s_2}^n$, the number of
terms contained in $\mathcal{S}_{1,1}$ is ``small'',  
while a further next step
will be based on estimations for the indices in $\mathcal{S}_{1,2}$.
Here we will observe that, as the letters are in different cells, we have enough independence (see the
decomposition in \eqref{covariancedecompose}) to show that the
contributions of the covariance terms from $\mathcal{S}_{1,2}$ are
``small".

Let us now focus on the first term on the right-hand side of
\eqref{est:S22}.
 Letting $g$ be as in
\eqref{defn:g}, and using arguments similar to those used in the
proof of Proposition~\ref{prop:combcomp}, we have,
\begin{eqnarray}\label{sum11}
\nonumber&&\sum_{\mathcal{S}_1} \frac{\left|Cov_{Z=1,
\mathcal{S}_{1,1}}(\Delta_j f(W) \Delta_jf(W^A), \Delta_k f(W)
\Delta_kf(W^B))\right|}{\binom{2n}{|A|}
(2n-|A|)\binom{2n}{|B|} (2n-|B|)} \\
\nonumber&& \qquad \le \mathbb{E} \left( \sum_{\mathcal{S}_1}
\frac{\left|g(A,B,j,k)\right| \mathbf{1}_{(A,B,j,k) \in
\mathcal{S}_{1,1}}}{\binom{2n}{|A|}(2n-|A|)\binom{2n}{|B|} (2n-|B|)} \Big| Z=1\right) \\
\nonumber && \qquad \le 4 \mathbb{E} \left( \sum_{\mathcal{S}_1}
\frac{\mathbf{1}_{(A,B,j,k) \in \mathcal{S}_{1,1}}}{\binom{2n}{|A|}
(2n-|A|)\binom{2n}{|B|}
(2n-|B|)}  \Big| Z=1\right) \\
&& \qquad = 4\mathbb{E}\left(|\mathcal{R}|\!\left|Z=1\right.\right),
\end{eqnarray}
where $$\mathcal{R} = \{(j,k) \in [2n]^2 : W_j\; \text{and} \; W_k
\; \text{are in the same cell of} \; \mathbf{r}\}.$$

To estimate \eqref{sum11}, for each $i=1, \dots, d$, let
$|\mathcal{R}_i|$ be the number of pairs of indices $(j,k) \in
[2n]^2$ that are in the $i$th-cell, and let ${\cal T}_i$ be the
event that $s_1n^{\alpha} \leq r_i - r_{i-1} \leq s_2 n^{\alpha}$.
Then,
\begin{eqnarray}\label{Ridecomp}
\!\!\!\!\mathbb{E}\left(|\mathcal{R}| \left|Z=1\right.\right)
&=&  \sum_{i=1}^{n^{1-\alpha}} \mathbb{E}(|\mathcal{R}_i| \left| Z=1\right.) \nonumber  \\
   &=& \sum_{i=1}^{n^{1-\alpha}}\!\mathbb{E}(|\mathcal{R}_i| \mathbf{1}_{{\cal T}_i} \left| Z=1\right.\!)
   + \sum_{i=1}^{n^{1-\alpha}}\! \mathbb{E}(|\mathcal{R}_i| \mathbf{1}_{{\cal T}_i^c}\left|
   Z=1\right.\!).
\end{eqnarray}
For the first term on the right-hand side of \eqref{Ridecomp}, note
that, when ${\cal T}_i$ holds true, the $X$-part of the $i$-th cell can contain at
most $n^{\alpha}$ letters while the $Y$-part can contain at most 
$s_2 n^{\alpha}$ ones.  Thus, 
$$|\mathcal{R}_i| \mathbf{1}_{{\cal T}_i}  \leq s_2 n^{2 \alpha},$$
and this leads to:  
\begin{equation}\label{Ridecompfirst}
    \sum_{i=1}^{n^{1-\alpha}} \mathbb{E}(|\mathcal{R}_i| \mathbf{1}_{{\cal T}_i} \big|
    Z=1) \leq
    s_2 n^{1+\alpha}.
\end{equation}

For the estimation of the second term on the right-hand side of
\eqref{Ridecomp}, we first  observe that  letting  $I := \{i \in [n^{1 - \alpha}] : 
{\cal T}_i \; \text{does not occur}\}$, we have 
$$\sum_{i=1}^{n^{1- \alpha}} \mathbb{E}(|\mathcal{R}_i| \mathbf{1}_{{\cal T}_i^c} \big| Z = 1) 
= \mathbb{E}\left( \sum_{i=1}^{n^{1- \alpha}}  |\mathcal{R}_i| \mathbf{1}_{{\cal T}_i^c} | Z = 1 \right)  
= \mathbb{E}\left( \sum_{i \in I}  |\mathcal{R}_i| \big| Z = 1 \right).$$  
Noting that $|\mathcal{R}_i| 
\leq 4 n^2$, we have 
$$\mathbb{E}\left( \sum_{i \in I}  |\mathcal{R}_i| \big| Z = 1 \right) 
\le 4 n^2 \mathbb{E} \left(\sum_{i\in I} \; 1 \; \big| Z=1 \right) 
\leq 4 n^2 \mathbb{E} \left(|I| \; \big| \; Z=1  \right).$$
Next, by definition, given that $Z = 1$, 
$|I| \le \epsilon n^{1 - \alpha}$ and so 
\begin{equation}\nonumber
\mathbb{E}\left(|I| \; \big| \; Z=1 \right) 
\le \epsilon n^{1 - \alpha} = 
c_1 \left(\frac{1 + \ln(1+ n^{\alpha})}{n^{\alpha}}\right)^{1/2} n^{1 - \alpha}
\leq Cn^{1 - {3 \alpha}/{2}} (\ln n^\alpha)^{1/2}.
\end{equation} 
Thus, we obtain 
$$\sum_{i=1}^{n^{1- \alpha}} \mathbb{E}(|\mathcal{R}_i| \mathbf{1}_{{\cal T}_i^c} \big| Z = 1) \leq 4n^2  \mathbb{E} \left(|I| \; \big| \; Z=1  \right) \leq C n^{3 - 3 \alpha /2} (\ln n^\alpha)^{1/2},$$
and when $\alpha > 2/3$, the above right-hand side is $o(n^2)$.

Hence,  using also \eqref{Ridecompfirst}, it follows that:  
\begin{eqnarray}\label{P22estimate}
\mathbb{E}(|\mathcal{R}|\!\left|Z=1\right.) \le C n^{1+\alpha} + C n^{3 - 3 \alpha /2} (\ln n^\alpha)^{1/2},
\end{eqnarray}
which, in turn, yields via \eqref{sum11},
\begin{eqnarray}\label{cov21estimate}
\sum_{\mathcal{S}_1} \frac{\left|Cov_{Z=1,
\mathcal{S}_{1,1}}(\Delta_j f(W) \Delta_jf(W^A), \Delta_k f(W)
\Delta_kf(W^B))\right|}{\binom{2n}{|A|}
(2n-|A|)\binom{2n}{|B|} (2n-|B|)}  \nonumber \\ 
\qquad \leq Cn^{1+\alpha}  + C n^{3 - 3 \alpha /2} (\ln n^\alpha)^{1/2}.
\end{eqnarray}
This last estimate takes care of the first sum on the 
right-hand side of \eqref{est:S22} as, again, this last right-hand side is  $o(n^2)$, when $2/3<\alpha < 1$. 

\

\noindent
{\it Hence, from here on, we henceforth assume that $\alpha$ is 
a real greater than $2/ 3$ and smaller than $1$.}

\

We move next to estimating the second term on the right-hand
side of \eqref{est:S22}, which is given by:  
$$\sum_{\mathcal{S}_1} \frac{Cov_{Z=1, \mathcal{S}_{1,2}}(\Delta_j f(W) \Delta_jf(W^A),
\Delta_k f(W) \Delta_kf(W^B))}{\binom{2n}{|A|}
(2n-|A|)\binom{2n}{|B|} (2n-|B|)}.$$ 
To estimate the summands in the above 
expression, we  decompose the covariance terms in such a
way that (conditional) independence of certain random variables occurs, therefore
simplifying the estimates themselves.  For this purpose, for each $i \in [2n]$,
let $f(P_i)=LC(P_i)$ be the length of the longest common subsequences
of $P_i^1$ and $P_i^2$, the coordinates of the cell
$P_i=(P_i^1;P_i^2)$.  Now, set
$$\tilde{\Delta}_if(W) := f(P_i) - f(P_i'),$$
where $P_i'$ is the same as $P_i$ except that $W_i$ is now replaced
with the independent copy $W_i'$.  In words, $\tilde{\Delta}_if(W)$
is the difference between the length of the longest common
subsequences of the two random words forming $P_i$, and the length of their modified versions at
coordinate $i$, i.e., the words forming $P_i'$.  Now for $(A,B,j,k) \in \mathcal{S}_1$,
\begin{eqnarray}\label{covariancedecompose}
&&\!\!\!\!\!\!\!Cov_{Z=1,
\mathcal{S}_{1,2}}(\Delta_jf(W)\Delta_jf(W^A),
\Delta_kf(W)\Delta_kf(W^B)) = \nonumber \\
&&\! Cov_{Z=1,
\mathcal{S}_{1,2}}((\Delta_jf(W)-\tilde{\Delta}_jf(W))
\Delta_jf(W^A),\Delta_kf(W)\Delta_kf(W^B)) \nonumber \\
&&\!\!\!+ Cov_{Z=1,
\mathcal{S}_{1,2}}(\tilde{\Delta}_jf(W)(\Delta_jf(W^A)-
\tilde{\tilde{\Delta}}_jf(W^A)),
\Delta_kf(W)\Delta_kf(W^B))
\nonumber   \\
&&\!\!\!+ Cov_{Z=1,
\mathcal{S}_{1,2}}(\tilde{\Delta}_jf(W)\tilde{\tilde{\Delta}}_jf(W^A),
(\Delta_kf(W)-\tilde{\Delta}_kf(W))\Delta_kf(W^B)) \nonumber \\
&&\!\!\!+ Cov_{Z=1,
\mathcal{S}_{1,2}}(\tilde{\Delta}_jf(W)\tilde{\tilde{\Delta}}_jf(W^A\!),\tilde{\Delta}_kf(W)
(\Delta_kf(W^B)-\tilde{\tilde{\Delta}}_kf(W^B))  \nonumber \\
&&\!\!\!+ Cov_{Z=1,
\mathcal{S}_{1,2}}(\tilde{\Delta}_jf(W)\tilde{\tilde{\Delta}}_jf(W^A),
\tilde{\Delta}_kf(W)
\tilde{\tilde{\Delta}}_kf(W^B)), 
\end{eqnarray}
where, for any $i\notin A$, we also set $\tilde{\tilde{\Delta}}_if(W^A) 
= f(W^A \big|_{P_i}) - f(W^{A\cup\{i\}} \big|_{P_i})$, 
with $W^A \big|_{P_i}$ and $W^{A\cup\{i\}} \big|_{P_i}$ being the restrictions of $W^A$ and 
$W^{A \cup \{i\}}$ to the cell $P_i$, respectively.  
Above, we used the bilinearity of $Cov_{Z=1, \mathcal{S}_{1,2}}$ to
express the left-hand side as a telescoping sum.  (Except for the
conditioning step, this decomposition is akin to a decomposition
developed in \cite{chat2}.)

Let us start by estimating the last term on the right-hand side of
\eqref{covariancedecompose}.  
 Letting 
$\xi_j:=\tilde{\Delta}_jf(W)\tilde{\tilde{\Delta}}_jf(W^A)$ and
$\xi_k:=\tilde{\Delta}_kf(W) \tilde{\tilde{\Delta}}_kf(W^B)$, with a slight abuse of notation as $\xi_j$ depends on $A$, while $\xi_k$ depends on $B$, we have 
\begin{eqnarray*}
&& \! \! \! \! \! \! \! \! \! \!  \! \! \! \! \!  Cov_{Z=1,
\mathcal{S}_{1,2}}(\tilde{\Delta}_jf(W)\tilde{\tilde{\Delta}}_jf(W^A),
\tilde{\Delta}_kf(W)
\tilde{\tilde{\Delta}}_kf(W^B))  \\
&&= \mathbb{E}((\xi_j - \mathbb{E} \xi_j) 
(\xi_k - \mathbb{E} \xi_k)  \mathbf{1}((A,B,j,k)\in \mathcal{S}_{1,2})| Z=1) \\
&&= \frac{\mathbb{E}((\xi_j - \mathbb{E} \xi_j) 
(\xi_k - \mathbb{E} \xi_k)  \mathbf{1}((A,B,j,k)\in \mathcal{S}_{1,2}, Z=1)) }{\mathbb{P}(Z=1)}  \\
&&= \frac{\mathbb{E}((\xi_j - \mathbb{E} \xi_j) 
(\xi_k - \mathbb{E} \xi_k)  \mathbf{1}((A,B,j,k)\in \mathcal{S}_{1,2})) }{\mathbb{P}(Z=1)} \\
&& \quad -  \frac{\mathbb{E}((\xi_j - \mathbb{E} \xi_j) 
(\xi_k - \mathbb{E} \xi_k)  \mathbf{1}((A,B,j,k)\in \mathcal{S}_{1,2}, Z = 0)) }{\mathbb{P}(Z=1)}  
\end{eqnarray*}
Note that the  second term on the right-most equality  is, in absolute value,  at most 
$$ C e^{-n^{1-\alpha}(1+\ln\,(1+n^\alpha))},$$ 
as from \eqref{Zequals0probest}, $\mathbb{P}(Z=1) \ge 1 -1/(2e)$ and $\mathbb{P}(Z=0) \le e^{-n^{1-\alpha}(1+\ln\,(1+n^\alpha))}$.  

So we focus on the other term and evaluate $\mathbb{E}((\xi_j - \mathbb{E} \xi_j) 
(\xi_k - \mathbb{E} \xi_k)  \mathbf{1}((A,B,j,k)\in \mathcal{S}_{1,2}))$. 
 We have 
\begin{eqnarray*} 
&&  \! \! \! \! \! \! \! \! \! \!  \! \! \! \! \!  \mathbb{E}((\xi_j - \mathbb{E} \xi_j) 
(\xi_k - \mathbb{E} \xi_k)  \mathbf{1}((A,B,j,k)\in \mathcal{S}_{1,2}))  \\
&& = \mathbb{E}((\xi_j - \mathbb{E} \xi_j) 
(\xi_k - \mathbb{E} \xi_k) \mid  \mathbf{1}((A,B,j,k)\in \mathcal{S}_{1,2}) \mathbb{P}((A,B,j,k)\in \mathcal{S}_{1,2}), 
\end{eqnarray*}
and, by conditional independence, the above right-hand side is equal to
$$
\mathbb{E}((\xi_j - \mathbb{E} \xi_j) 
  \mid  \mathbf{1}((A,B,j,k)\in \mathcal{S}_{1,2}))) \mathbb{E}((\xi_k - \mathbb{E} \xi_k) 
  \mid  \mathbf{1}((A,B,j,k)\in \mathcal{S}_{1,2})))  \mathbb{P}((A,B,j,k)\in \mathcal{S}_{1,2}).  
$$

Now,
\begin{eqnarray*}
&&  \! \! \! \! \! \! \! \! \! \!  \! \! \! \! \!  \mathbb{E}((\xi_j - \mathbb{E} \xi_j) 
  \mid  \mathbf{1}((A,B,j,k)\in \mathcal{S}_{1,2}))  \\
  &=	& \mathbb{E}((\xi_j
  \mid  \mathbf{1}((A,B,j,k)\in \mathcal{S}_{1,2}))  - \mathbb{E} \xi_j  \\
  &=& \mathbb{E}((\xi_j
  \mid  \mathbf{1}((A,B,j,k)\in \mathcal{S}_{1,2})) \\
  && - \  \mathbb{E}((\xi_j
  \mid  \mathbf{1}((A,B,j,k)\in \mathcal{S}_{1,2})) \mathbb{P}( (A,B,j,k)\in \mathcal{S}_{1,2}) \\
  && - \ \mathbb{E}((\xi_j
  \mid  \mathbf{1}((A,B,j,k)\in \mathcal{S}_{1,1})) \mathbb{P}(  (A,B,j,k)\in \mathcal{S}_{1,1}).  
\end{eqnarray*}
Using elementary manipulations, the last equality can be rewritten as 
$$ (\mathbb{E}((\xi_j
  \mid  \mathbf{1}((A,B,j,k)\in \mathcal{S}_{1,2})) - \mathbb{E}((\xi_j
  \mid  \mathbf{1}((A,B,j,k)\in \mathcal{S}_{1,1})) ) \mathbb{P}(   (A,B,j,k)\in \mathcal{S}_{1,1}).$$
 Following exactly the same steps,    write $\mathbb{E}((\xi_k - \mathbb{E} \xi_k) 
  \mid  \mathbf{1}((A,B,j,k)\in \mathcal{S}_{1,2}))$ as  
  $$ (\mathbb{E}((\xi_k
  \mid  \mathbf{1}((A,B,j,k)\in \mathcal{S}_{1,2})) - \mathbb{E}((\xi_k
  \mid  \mathbf{1}((A,B,j,k)\in \mathcal{S}_{1,1})) ) \mathbb{P}(   (A,B,j,k)\in \mathcal{S}_{1,1}).$$

Combining these observations,  and using again  \eqref{Zequals0probest}, $\mathbb{P}(Z=1) \geq 1 - 1/(2e)$,  $n \geq 1$,   we obtain
   \begin{eqnarray*} 
&&   \! \! \! \! \! \! \! \! \! \!  \! \! \! \! \!  Cov_{Z=1,
\mathcal{S}_{1,2}}(\tilde{\Delta}_jf(W)\tilde{\tilde{\Delta}}_jf(W^A),
\tilde{\Delta}_kf(W)
\tilde{\tilde{\Delta}}_kf(W^B))  \\
&& \leq \frac{C \mathbb{P}((A,B,j,k)\in \mathcal{S}_{1,2}) (\mathbb{P}((A,B,j,k)\in \mathcal{S}_{1,1}))^2 }{\mathbb{P}(Z=1)} + C \mathbb{P}(Z = 0) \\
&& \leq C \mathbb{P}((A,B,j,k)\in \mathcal{S}_{1,1}) + C \mathbb{P}(Z = 0) \\
&& \leq C (\mathbb{P}((A,B,j,k)\in \mathcal{S}_{1,1}, Z=1) + \mathbb{P}((A,B,j,k)\in \mathcal{S}_{1,1}, Z=0)) +  C \mathbb{P}(Z = 0) \\
&& \leq C \mathbb{P}((A,B,j,k)\in \mathcal{S}_{1,1} \mid Z=1) \mathbb{P}(Z=1)   + C \mathbb{P}(Z = 0).
\end{eqnarray*}

We therefore arrive at:  
\begin{eqnarray}\label{estfirstinthesumlat}
&&\!\!\!\!\!\!\!\!\!\!\!\!\!\!\!\!\!\!\!\sum_{\mathcal{S}_1  } \frac{\left|Cov_{Z=1,
\mathcal{S}_{1,2}}(\tilde{\Delta}_jf(W)\tilde{\tilde{\Delta}}_jf(W^A),\tilde{\Delta}_kf(W)
\tilde{\tilde{\Delta}}_kf(W^B))\right|}{\binom{2n}{|A|}
(2n-|A|)\binom{2n}{|B|} (2n-|B|)} \nonumber 
\nonumber \\ 
&& \leq 
\sum_{\mathcal{S}_1 } \frac{C \mathbb{P}((A,B,j,k)\in \mathcal{S}_{1,1}|Z=1) \mathbb{P}(Z=1)  
+ C \mathbb{P}(Z=0)}{\binom{2n}{|A|}
(2n-|A|)\binom{2n}{|B|} (2n-|B|)}
\nonumber \\ 
&& \leq 
\sum_{\mathcal{S}_1 } \frac{C \mathbb{E}(\mathbf{1}((A,B,j,k)\in \mathcal{S}_{1,1})|Z=1)  
+ C \mathbb{P}(Z=0)}{\binom{2n}{|A|}
(2n-|A|)\binom{2n}{|B|} (2n-|B|)}
\nonumber \\ 
&&  \leq
C n^{1+\alpha} +   C n^2 e^{-n^{1-\alpha}(1+\log\,(1+n^\alpha))},
\end{eqnarray}
where for the last step \eqref{P22estimate} is used, as well as the estimates 
in \eqref{MasterEq} and \eqref{Zequals0probest}.

We continue by obtaining upper bounds for the first four summands in \eqref{covariancedecompose}. 
We just focus on the estimation of the first of these four terms since the other three can be 
estimated in a similar way.  
Indeed, it will be clear from the discussion below that the third of these four terms 
can be estimated in exactly the same way as done for the first of the four.  
Also, with steps similar to the ones  performed in estimating this first term, 
one can easily see that the estimation of the second and fourth of these terms reduces to the estimation of 
$$\mathbb{E}_{Z=1, \mathcal{S}_{1,2}} 
|\Delta_j f(W^A) - \tilde{\tilde{\Delta}}_j f(W^A)|.$$ 
(Again, and throughout,  $\mathbb{E}_{Z=1}$ is short for conditional expectation given $\{Z=1\}$, 
while $\mathbb{E}_{Z=1, \mathcal{S}_{1,2}}(\cdot) = \mathbb{E}_{Z=1}(\cdot\mathbf{1}_{(A,B,j,k) 
\in \mathcal{S}_{1,2}})$.)  
 Next,  
\begin{eqnarray*}
&&\!\!\!\!\!\!\!\!\!\!\!\!\!\!\!\!\!\!\! \mathbb{E}_{Z=1, \mathcal{S}_{1,2}} 
|\Delta_j f(W^A) - \tilde{\tilde{\Delta}}_j f(W^A)| \\ &=& \mathbb{E} \left(
|\Delta_j f(W^A) - \tilde{\tilde{\Delta}}_j f(W^A)| \mathbf{1}((A,B,j,k) \in \mathcal{S}_{1,2}) \big| Z=1 \right) \\ 
 &=& \frac{\mathbb{E} \left(
|\Delta_j f(W^A) - \tilde{\tilde{\Delta}}_j f(W^A)| \mathbf{1}((A,B,j,k)
 \in \mathcal{S}_{1,2})\mathbf{1}(Z=1) \right)}{\mathbb{P}(Z=1)} \\
&\leq& \frac{\mathbb{E} \left(
|\Delta_j f(W^A) - \tilde{\tilde{\Delta}}_j f(W^A)| \mathbf{1}((A,B,j,k) \in \mathcal{S}_{1,2}) \right)}{\mathbb{P}(Z=1)}. 
\end{eqnarray*}
Now, writing $\mathcal{S}_{1,2}^A$ in place of $\mathcal{S}_{1,2}$ when using 
the sequence $W^A$ instead of $W$, 
the last inequality, just above, leads to:   
\begin{eqnarray}\label{eq:covdecompotherterms}
\nonumber &&\!\!\!\!\!\!\!\!\!\!\!\!\!\!\!\!\!\!\! \mathbb{E}_{Z=1, \mathcal{S}_{1,2}} 
|\Delta_j f(W^A) - \tilde{\tilde{\Delta}}_j f(W^A)|  \\ \nonumber  &\leq& \frac{\mathbb{E} \left(
|\Delta_j f(W^A) - \tilde{\tilde{\Delta}}_j f(W^A)| \right)}{\mathbb{P}(Z=1)} \\
\nonumber  &=&  \frac{\mathbb{E} \left(
|\Delta_j f(W^A) - \tilde{\tilde{\Delta}}_j f(W^A)|  \mathbf{1}((A,B,j,k) 
\in \mathcal{S}_{1,2}^A) \right)}{\mathbb{P}(Z=1)}  \\
&& +  \frac{\mathbb{E} \left(
|\Delta_j f(W^A) - \tilde{\tilde{\Delta}}_j f(W^A)|  \mathbf{1}((A,B,j,k) 
\notin \mathcal{S}_{1,2}^A) \right)}{\mathbb{P}(Z=1)}. 
\end{eqnarray}
Then, since 
\begin{eqnarray*}
&\!\!\!\!\!\!\!|\Delta_j f(W^A) - \tilde{\tilde{\Delta}}_j f(W^A)|  \mathbf{1}((A,B,j,k) \in \mathcal{S}_{1,2}^A) \\
&\qquad \qquad \qquad \qquad =_d |\Delta_j f(W) - \tilde{\Delta}_j f(W)| \mathbf{1}((A,B,j,k) \in \mathcal{S}_{1,2}), 
\end{eqnarray*}
the first term on the right-hand side of 
\eqref{eq:covdecompotherterms} is equal to 
\begin{equation}\label{eq:covdecompothertermsadd1}
\frac{\mathbb{E}\left(
|\Delta_j f(W) - \tilde{\Delta}_j f(W)|\mathbf{1}((A,B,j,k) \in \mathcal{S}_{1,2})\right)}{\mathbb{P}(Z=1)},
\end{equation}
which we will estimate further, below, when working out the estimation of the first term on the right-hand side 
of \eqref{eq:covdecompotherterms}.  
Also, for the second term in \eqref{eq:covdecompotherterms}, 
noting that $|\Delta_j f(W) - \tilde{\tilde{\Delta}}_j f(W)| \leq 2$, and that by the iid assumption 
$W$ and $W^A$ are 
identically distributed, we have 
\begin{eqnarray}
\nonumber &&\!\!\!\!\!\!\!\!\!\!\!\!\!\!\!\!\!\!\! \frac{\mathbb{E} \left(
|\Delta_j f(W^A) - \tilde{\tilde{\Delta}}_j f(W^A)|  \mathbf{1}((A,B,j,k) 
\notin \mathcal{S}_{1,2}^A) \right)}{\mathbb{P}(Z=1)}  \\ \nonumber  && = \frac{\mathbb{E} \left(
|\Delta_j f(W) - \tilde{\tilde{\Delta}}_j f(W)|  \mathbf{1}((A,B,j,k) 
\notin \mathcal{S}_{1,2}) \right)}{\mathbb{P}(Z=1)} \\
\nonumber  && \leq C \mathbb{E}(\mathbf{1}_{(A,B,j,k) \in \mathcal{S}_{1,1}} \mid Z=1),
\end{eqnarray}
and then  
\begin{eqnarray}
\nonumber &&\!\!\!\!\!\!\!\!\!\!\!\!\!\!\!\!\!\!\! \sum_{\mathcal{S}_1} \frac{\mathbb{E} \left(
|\Delta_j f(W^A) - \tilde{\tilde{\Delta}}_j f(W^A)|  \mathbf{1}((A,B,j,k) 
\notin \mathcal{S}_{1,2}^A) \right)}{\mathbb{P}(Z=1) \binom{2n}{|A|}
(2n-|A|)\binom{2n}{|B|} (2n-|B|)}   \\
\nonumber  && \leq C \sum_{\mathcal{S}_1} \frac{\mathbb{E}(\mathbf{1}_{(A,B,j,k) \in \mathcal{S}_{1,1}} \mid Z=1)}{  \binom{2n}{|A|}
(2n-|A|)\binom{2n}{|B|} (2n-|B|)}.
\end{eqnarray}
But, this last term on the right-hand side was already shown to be bounded by 
$ Cn^{1+\alpha}  + C n^{3 - 3 \alpha /2} (\ln n^\alpha)^{1/2}$, while reaching out  \eqref{cov21estimate}.
Therefore, focusing on the estimation of  ${\mathbb{E}_{\mathcal{S}_{1,2}}
|\Delta_j f(W) - \tilde{\Delta}_j f(W)|}/{\mathbb{P}(Z=1)}$ or, indeed, merely on the estimation of $\mathbb{E}_{\mathcal{S}_{1,2}}
|\Delta_j f(W) - \tilde{\Delta}_j f(W)|$, will suffice for our purposes for the  second and the 
fourth of the terms in  \eqref{covariancedecompose}. 
 This will be done while discussing the estimation of the first term below as noted earlier. 
 
So,	 we can now focus on estimating the first term in \eqref{covariancedecompose} (and as already indicated, 
similar arguments will 
provide a similar estimate for the other three terms) which is given by:  $$Cov_{Z=1,
\mathcal{S}_{1,2}}((\Delta_jf(W)-\tilde{\Delta}_jf(W))
\Delta_jf(W^A),\Delta_kf(W)\Delta_kf(W^B)).$$ 
To do so, let $$U :=
(\Delta_jf(W)-\tilde{\Delta}_jf(W))\Delta_jf(W^A),$$ and
$$V:=\Delta_kf(W)\Delta_kf(W^B),$$
so that we wish to estimate $Cov_{Z=1, \mathcal{S}_{1,2}}(U,V)$.
But,
\begin{eqnarray*}
&&\!\!\!\!\!\!\left| Cov_{Z=1, \mathcal{S}_{1,2}}(U,V)\right| \\
&&   =\left|\mathbb{E}((U-\mathbb{E}U)(V-\mathbb{E}V)
\mathbf{1}_{(A,B,j,k) \in \mathcal{S}_{1,2}}\!\left|\right.\!Z=1)\right| \\
&& \le \mathbb{E}(|UV| \mathbf{1}_{(A,B,j,k) \in
\mathcal{S}_{1,2}}\!\left|\right.\!Z=1)+ \mathbb{E}|V|
\mathbb{E}(|U|\mathbf{1}_{(A,B,j,k) \in \mathcal{S}_{1,2}} \!\left|\right.\!Z=1) \\
&&  \qquad  + \  \mathbb{E}|U| \mathbb{E}(|V| \mathbf{1}_{(A,B,j,k) \in
\mathcal{S}_{1,2}}\!\left|\right.\!Z=1)
+ \mathbb{E}|U|\mathbb{E}|V|\mathbb{E}
(\mathbf{1}_{(A,B,j,k) \in \mathcal{S}_{1,2}} \!\left|\right.\!Z=1) \\
&& := T_1+T_2+T_3+T_4,
\end{eqnarray*}
and note here that $T_i, i=1,2,3,4$ are functions of $(A,B,j,k)$. Let us
begin by estimating
$$T_1 = \mathbb{E}_{Z=1}|(\!(\Delta_jf(W)-\tilde{\Delta}_jf(W)\!)
\Delta_jf(W^A)\!)(\Delta_kf(W)\Delta_kf(W^B)\!)\mathbf{1}_{(A,B,j,k)
\in \mathcal{S}_{1,2}}|.$$ Since
$|\Delta_jf(W^A)(\Delta_kf(W)\Delta_kf(W^B))| \leq 1$,
\begin{equation}\label{est:T1}
T_1 \le \mathbb{E}_{Z=1}\left(|\Delta_jf(W)-\tilde{\Delta}_jf(W)|
\mathbf{1}_{(A,B,j,k) \in \mathcal{S}_{1,2}}\right).
\end{equation}
A similar estimate also reveals that
\begin{equation}\label{est:T2}
T_2 \le
\mathbb{E}_{Z=1}\left(|\Delta_jf(W)-\tilde{\Delta}_jf(W)|\mathbf{1}_{(A,B,j,k)
\in \mathcal{S}_{1,2}}\right).
\end{equation}
Next, for $T_3$ and $T_4$, and since $|V| \le 1$,
\begin{eqnarray}\label{est:T3T4}
\nonumber T_3 + T_4 \leq 2 \mathbb{E}|U| 
&\leq& 2 \mathbb{E}|\Delta_jf(W)-\tilde{\Delta}_jf(W)| \\
\nonumber &=& 2
\mathbb{E}_{Z=1}\left(|\Delta_jf(W)-\tilde{\Delta}_jf(W)|
\mathbf{1}_{(A,B,j,k) \in \mathcal{S}_{1,2}}\right)\mathbb{P}(Z=1) \\
\nonumber && \!\!\!\!+ 2
\mathbb{E}_{Z=1}\left(|\Delta_jf(W)-\tilde{\Delta}_jf(W)|
\mathbf{1}_{(A,B,j,k) \in \mathcal{S}_{1,1}}\right)\mathbb{P}(Z=1)\\
\nonumber && \!\!\!\!+ 2
\mathbb{E}_{Z=0}\left(|\Delta_jf(W)-\tilde{\Delta}_jf(W)|
\mathbf{1}_{(A,B,j,k) \in \mathcal{S}_{1,2}}\right)\mathbb{P}(Z=0)\\
\nonumber && \!\!\!\!+ 2
\mathbb{E}_{Z=0}\left(|\Delta_jf(W)-\tilde{\Delta}_jf(W)|
\mathbf{1}_{(A,B,j,k) \in \mathcal{S}_{1,1}}\right)\mathbb{P}(Z=0) \\
\nonumber &\leq&  2
\mathbb{E}_{Z=1}\left(|\Delta_jf(W)-\tilde{\Delta}_jf(W)|
\mathbf{1}_{(A,B,j,k) \in \mathcal{S}_{1,2}}\right)\\
\nonumber && \!\!\!\!  + 2
\mathbb{E}_{Z=1}\left(|\Delta_jf(W)-\tilde{\Delta}_jf(W)|
\mathbf{1}_{(A,B,j,k) \in \mathcal{S}_{1,1}}\right) \\
&& \!\!\!\! + Ce^{-n^{1-\alpha}(1+\ln\,(1+n^\alpha))},
\end{eqnarray}
where $\mathbb{E}_{Z=0}$ (resp.~$\mathbb{E}_{Z=1}$) is short for the conditional expectation given $\{Z=0\}$ (resp.~given $\{Z=1\}$), and 
where we used the trivial bound on $\mathbb{P}(Z=1)$, and also
\eqref{Zequals0probest}, for the last inequality.

Now, denote by $h(A,B,j,k)$ the sum of the first four terms on the
right-hand side of \eqref{covariancedecompose}.  Then, performing
estimations as in getting \eqref{est:T1}, \eqref{est:T2} and
\eqref{est:T3T4}, for the first  and third term of this sum, 
and keeping in mind the discussion following  \eqref{eq:covdecompotherterms},  
so that similar estimates also hold true for the second and fourth term of the sum,   
we obtain
 \begin{eqnarray*}
&&\! \sum_{\mathcal{S}_1} \left|\frac{h(A,B,j,k)}{\binom{2n}{|A|}
(2n-|A|)\binom{2n}{|B|} (2n-|B|)}\right| \nonumber \\
&& \quad \leq C \sum_{\mathcal{S}_1}
\frac{\mathbb{E}_{Z=1}\left(|\Delta_jf(W)-\tilde{\Delta}_jf(W)|
\mathbf{1}_{(A,B,j,k) \in \mathcal{S}_{1,2}}\right)}{\binom{2n}{|A|}
(2n-|A|)\binom{2n}{|B|} (2n-|B|)}\\
&& \quad  \quad + C\sum_{\mathcal{S}_1}
\frac{\mathbb{E}_{Z=1}\left(|\Delta_jf(W)-\tilde{\Delta}_jf(W)|
\mathbf{1}_{(A,B,j,k) \in \mathcal{S}_{1,1}}\right)}{\binom{2n}{|A|}
(2n-|A|)\binom{2n}{|B|} (2n-|B|)} \\
&& \quad \quad + C\sum_{\mathcal{S}_1}  \frac{\mathbb{E}_{Z=1}\left(|\Delta_k
f(W)-\tilde{\Delta}_k f(W)| \mathbf{1}_{(A,B,j,k) \in
\mathcal{S}_{1,2}}\right)}{\binom{2n}{|A|}
(2n-|A|)\binom{2n}{|B|} (2n-|B|)}\\
&& \quad \quad  + C\sum_{\mathcal{S}_1}  \frac{\mathbb{E}_{Z=1}\left(|\Delta_k
f(W)-\tilde{\Delta}_k f(W)| \mathbf{1}_{(A,B,j,k) \in
\mathcal{S}_{1,1}}\right)}{\binom{2n}{|A|}
(2n-|A|)\binom{2n}{|B|} (2n-|B|)} \\
&& \quad \quad+ C\sum_{\mathcal{S}_1} \frac{e^{-n^{1-\alpha}
(1+\ln\,(1+n^\alpha))}}{\binom{2n}{|A|} (2n-|A|)\binom{2n}{|B|} (2n-|B|)}.
\end{eqnarray*}
Noting that the sums involving $k$'s are identical to the sums  involving $j$'s,  we rewrite this last upper bound as  
\begin{eqnarray*}
&&\! \sum_{\mathcal{S}_1} \left|\frac{h(A,B,j,k)}{\binom{2n}{|A|}
(2n-|A|)\binom{2n}{|B|} (2n-|B|)}\right| \nonumber \\
&& \quad \leq C \sum_{\mathcal{S}_1}
\frac{\mathbb{E}_{Z=1}\left(|\Delta_jf(W)-\tilde{\Delta}_jf(W)|
\mathbf{1}_{(A,B,j,k) \in \mathcal{S}_{1,2}}\right)}{\binom{2n}{|A|}
(2n-|A|)\binom{2n}{|B|} (2n-|B|)}\\
&& \quad \quad  + C\sum_{\mathcal{S}_1}
\frac{\mathbb{E}_{Z=1}\left(|\Delta_jf(W)-\tilde{\Delta}_jf(W)|
\mathbf{1}_{(A,B,j,k) \in \mathcal{S}_{1,1}}\right)}{\binom{2n}{|A|}
(2n-|A|)\binom{2n}{|B|} (2n-|B|)} \\
&& \quad \quad+ C\sum_{\mathcal{S}_1} \frac{e^{-n^{1-\alpha}(1+\ln\,(1+n^\alpha))}}
{\binom{2n}{|A|} (2n-|A|)\binom{2n}{|B|} (2n-|B|)}.
\end{eqnarray*}
As with previous computations, using \eqref{MasterEq} 
and \eqref{Zequals0probest}, the
third sum on the above right-hand side is itself upper-bounded by 
\begin{equation}\label{S11first}
Cn^2e^{-n^{1-\alpha}(1+\ln\,(1+n^\alpha))},
\end{equation}
while, using \eqref{sum11} and \eqref{P22estimate}, the middle sum is 
upper-bounded by
\begin{equation}\label{S11second}
Cn^{1+\alpha}.  
\end{equation}

Therefore, we are just left with estimating
\begin{equation*}\label{cov22estimate2}
\sum_{\mathcal{S}_1}  \frac{\mathbb{E}_{Z=1}\left(|\Delta_jf(W)-\tilde{\Delta}_jf(W)|
\mathbf{1}_{(A,B,j,k) \in \mathcal{S}_{1,2}}\right)}{\binom{2n}{|A|}
(2n-|A|)\binom{2n}{|B|} (2n-|B|)}.
\end{equation*}

Noting that 
\begin{eqnarray}\label{est:cov22ctd}
\nonumber \sum_{\mathcal{S}_1}  \frac{\mathbb{E}_{Z=1}\left(|\Delta_jf(W)-\tilde{\Delta}_jf(W)|
\mathbf{1}_{(A,B,j,k) \in \mathcal{S}_{1,2}}\right)}{\binom{2n}{|A|}
(2n-|A|)\binom{2n}{|B|} (2n-|B|)} \\
\le \quad \sum_{\mathcal{S}_1}  \frac{\mathbb{E}_{Z=1}|\Delta_jf(W)-\tilde{\Delta}_jf(W)|}{\binom{2n}{|A|}
(2n-|A|)\binom{2n}{|B|} (2n-|B|)},
\end{eqnarray}
we can just focus on estimating $\mathbb{E}_{Z=1}|\Delta_jf(W)-\tilde{\Delta}_jf(W)|$. 
To do so, the following 
simple proposition will be   useful.  
\begin{proposition}\label{propn:monotone}
For any $j \in [2n]$, $$\Delta_j f(W) \leq \widetilde{\Delta}_j f(W).$$
\end{proposition}

\begin{proof}
Assume not, and that $\Delta_j f(W) > \widetilde{\Delta}_j f(W).$
Then either $\Delta_j f(W) =1$  and $\widetilde{\Delta}_j f(W)=0$,
or $\Delta_j f(W)=0$ and $\widetilde{\Delta}_j f(W)=-1$. Consider
the former. Then, changing the $j$th coordinate does not affect the
length of the longest common subsequence of the cell containing $j$.
Since the coordinates outside that particular cell have not been
changed, the overall length of the longest common subsequence 
cannot decrease, that is, $\Delta_j$ cannot be $1$. 
The other case is similar.  
\end{proof}

Returning to the estimation of 
$\mathbb{E}_{Z=1}|\Delta_jf(W)-\tilde{\Delta}_jf(W)|$, 
using the domination property obtained  in  
Proposition~\ref{propn:monotone}, we have 
$$\mathbb{E}_{Z=1}|\Delta_jf(W)-\tilde{\Delta}_jf(W)|
 = \mathbb{E}_{Z=1}(\tilde{\Delta}_jf(W)) -  \mathbb{E}_{Z=1}(\Delta_jf(W)).$$

We now claim that both terms on the right-hand side of the last expression 
are exponentially small in $n$.  Let us first deal with 
$\mathbb{E}_{Z=1}(\Delta_jf(W))$, the other term, which is similar, is dealt with afterwards. 

We have 
\begin{eqnarray*}
\mathbb{E}_{Z=1}(\Delta_jf(W)) &=& \mathbb{E}_{Z=1}(\Delta_jf(W) \mathbf{1}(Z^j =1)) 
+ \mathbb{E}_{Z=1}(\Delta_jf(W) \mathbf{1}(Z^j =0)) \\
&=& \frac{\mathbb{E}(\Delta_jf(W) \mathbf{1}(Z =1) 
\mathbf{1}(Z^j =1))}{\mathbb{P}(Z=1)} \\
&& + \quad \frac{\mathbb{E}(\Delta_jf(W) \mathbf{1}(Z =1) \mathbf{1}(Z^j =0))}
{\mathbb{P}(Z=1)},
\end{eqnarray*}
where $Z^j$ is the indicator random variable 
defined in the same way as $Z$, except that the $j^{th}$ coordinate of $W$ 
is replaced by the independent copy $W_j'$.  
Note that, for 
any $j \in [2n]$, $Z$ and $Z^j$ are identically distributed 
but that they are certainly not independent. 

Looking, first, at the second term in the last expression, we have, with 
the help of \eqref{Zequals0probest}, and since $Z$ and $Z^j$ are identically distributed, 

\begin{eqnarray*}
\frac{\left|\mathbb{E}(\Delta_jf(W) \mathbf{1}(Z =1) 
\mathbf{1}(Z^j =0))\right|}{\mathbb{P}(Z=1)} 
\leq \frac{\mathbb{P}(Z^j = 0)}{\mathbb{P}(Z = 1)} \le C e^{-n^{1-\alpha
	}(1+\ln\,(1+n^\alpha))}. 
\end{eqnarray*}
Also, writing 
\begin{eqnarray*}
\mathbb{E}(\Delta_jf(W) \mathbf{1}(Z =1) \mathbf{1}(Z^j =1)) 
&=& \mathbb{E}((f(W)-f(W^j)) \mathbf{1}(Z =1) \mathbf{1}(Z^j =1)) \\
&=& \mathbb{E}(f(W) \mathbf{1}(Z =1) \mathbf{1}(Z^j =1))  \\
&& - \  \mathbb{E}(f(W^j) \mathbf{1}(Z =1) \mathbf{1}(Z^j =1))\\
&=& 0, 
\end{eqnarray*}
\noindent
since, again, $Z$ and   $Z^j$ are identically distributed.  
These observations yield
$$\left|\mathbb{E}_{Z=1}(\Delta_jf(W))\right| 
\le C e^{-n^{1-\alpha}(1+\ln\,(1+n^\alpha))}.$$ 
Similarly, noting that the expectation is conditional on $Z=1$, replacing $n$ by $n^{\alpha}$, we have  
$$  |\mathbb{E}_{Z=1}(\tilde{\Delta}_j f(W) \mathbf{1}((A,B,j,k) \in \mathcal{S}_{1,2})) | \leq  C e^{- n^{(1 - \alpha ) \alpha} (1 + \log (1+n^{\alpha^2}))}.$$ 
(The reason for this last inequality is the fact that the configurations belong to $\mathcal{S}_{1,2}$ and, in that case, we just deal with a scaled version of the LCS problem.)

Now, note that 
\begin{eqnarray*}
\left|\mathbb{E}_{Z=1}(\tilde{\Delta}_jf(W))\right|  &\leq&  |\mathbb{E}_{Z=1}(\tilde{\Delta}_j f(W)  \mathbf{1}((A,B,j,k) \in \mathcal{S}_{1,1}) )|, \\
&& + \  |\mathbb{E}_{Z=1}(\tilde{\Delta}_j f(W) \mathbf{1}((A,B,j,k) \in \mathcal{S}_{1,2})) |,
\end{eqnarray*}
and, via Proposition ~\ref{prop:combcomp},  
$$\sum_{\mathcal{S}_1} \frac{ |\mathbb{E}_{Z=1}(\tilde{\Delta}_j f(W)  \mathbf{1}((A,B,j,k) \in \mathcal{S}_{1,1})) |}{\binom{2n}{|A|}
(2n-|A|)\binom{2n}{|B|} (2n-|B|)} \leq C n^{1 + \alpha},$$ 
and 
 $$\sum_{\mathcal{S}_1} \frac{ |\mathbb{E}_{Z=1}(\tilde{\Delta}_j f(W) \mathbf{1}((A,B,j,k) \in \mathcal{S}_{1,2})) |}{\binom{2n}{|A|}
(2n-|A|)\binom{2n}{|B|} (2n-|B|)} \leq C n^2  e^{- n^{(1 - \alpha) \alpha} (1 + \log (1+n^{\alpha^2}))},$$ which, when combined, yields 
 $$\sum_{\mathcal{S}_1} \frac{ |\mathbb{E}_{Z=1}(\tilde{\Delta}_j f(W))|}{\binom{2n}{|A|}
(2n-|A|)\binom{2n}{|B|} (2n-|B|)} \leq C n^{1 + \alpha}.$$ 
Thus,  from \eqref{est:cov22ctd} and the above estimates, 
\begin{equation}\label{est:last}
\sum_{\mathcal{S}_1}  \frac{\mathbb{E}_{Z=1}\left(|\Delta_jf(W)-\tilde{\Delta}_jf(W)|
\mathbf{1}_{(A,B,j,k) \in \mathcal{S}_{1,2}}\right)}{\binom{2n}{|A|}
(2n-|A|)\binom{2n}{|B|} (2n-|B|)} \leq C n^{1 + \alpha}.\end{equation}

\noindent 
 Combining \eqref{est:sum20},  \eqref{cov21estimate},  \eqref{estfirstinthesumlat},   
 \eqref{S11first},  \eqref{S11second} and \eqref{est:last} 
finally gives
\begin{eqnarray}\label{Lastvarbound}
\Var T &\le& C\left(n^2 e^{-n^{1-\alpha }(1+\ln\,(1+n^\alpha))} +
n^{1+\alpha}+n^{3 - 3 \alpha /2} (\ln n^\alpha)^{1/2}\right).
\end{eqnarray}
Therefore, Theorem~\ref{chat1} and \eqref{est:2ndterm}, 
ensure that:
\begin{eqnarray*}\label{chatbound3}
d_W\left(\frac{LC_n-\bbe LC_n}{\sqrt{\Var LC_n}},\mathcal{G}\right)
\le
C\frac{1}{n^{\frac{1-\alpha}{2}}},
\end{eqnarray*}
holds for every $n\geq 1$, with $C>0$ a constant independent of $n$, and for $\alpha > 4/5$ as then 
$1+\alpha > 3 -3\alpha/2$. 
\hfill $\square$

\begin{remark}\label{rmk:unifcase}

(i) The constant $C$ in Theorem~\ref{thm: CLT} is independent of
$n$ but depends on $m$, on $\alpha$, on $s_1$ and $s_2$ of Theorem~\ref{thm:
HMdiag}, which in turn depend on the distribution of $X_1$, as well as on 
the quantities involved in the constant $K$ 
and $C$ in \eqref{HMavarestimate}--\eqref{est:2ndterm}.
 
(ii) Of course, there is no reason for our rate 
$1/{n^{(1-\alpha)/2}}$ to be
sharp (as previously mentioned, for $2/3<\alpha<4/5$, the rate $1/n^{1-3(1-\alpha/2)/2}$ 
is possible).  Also, instead of the choice $v=n^{\alpha}$, a choice such
as $v = h(n)$, for some optimal function $h$ would improve this rate.  
Can we conjecture that the optimal rate in Kolmogorov distance is 
$1/\sqrt n$?    

(iii) From a known duality between the length of a longest common
subsequence of two random words and the length of a shortest
common supersequence (see Dan$\breve{\rm c}$\'ik~\cite{D}), our
result also implies a central limit theorem for this latter case.

\end{remark}

\section{Concluding Remarks}

We conclude the paper with a brief discussion on longest common
subsequences in random permutations and, in a final remark, present
some potential extensions, perspectives and related questions we believe are
of interest.

Theorem~\ref{thm: CLT} shows that the Gaussian distribution appears
as the limiting law for the length of the longest common subsequences of two 
random words.  However, the Tracy-Widom distribution has also been 
hypothesized as the limiting law in sequence comparison 
problems, e.g., \cite{AD}.  It turns out, as 
shown next, that it is indeed the case for certain distributions on
permutations.

First, it is folklore that, if $\pi=(\pi_1,\ldots,\pi_n)$ is any 
element of the symmetric group $\frak S_n$, then
\begin{equation}\label{lislcs1}
LI_n(\pi) = LC_n((1,2,\dots,n),(\pi_1,\pi_2,\dots,\pi_n)),
\end{equation}
where $LI_n(\pi)$ is the length of the longest increasing
subsequence in $\pi=(\pi_1,\dots,\pi_n)$, while
$LC_n((1,2,\ldots,n),(\pi_1,\pi_2,\dots,\pi_n))$, is the length of
the longest common subsequence of the identity permutation $id$ and
of the permutation $\pi$.  In the equality \eqref{lislcs1},
replacing $id$ with an arbitrary permutation $\rho$ and taking for
$\pi$ a uniform random permutation in $\frak S_n$ lead to:

\begin{proposition}
(i) Let $\rho=(\rho_1,\rho_2,\dots,\rho_n)$ be a fixed permutation
in $\frak S_n$ and let $\pi$ be a uniform random permutation in $\frak S_n$.
Then,
\begin{equation}\label{lislcs2}
LI_n(\pi) =_d
LC_n((\rho_1,\rho_2,\ldots,\rho_n),(\pi_1,\pi_2,\dots,\pi_n)),
\end{equation}
where $=_d$ denotes equality in distribution.

(ii) Let $\rho$ and $\pi$ be two independent uniform random
permutations in $\frak S_n$, and let $x \in \mathbb{R}$.  Then,
\begin{equation}\label{eq:distributionstep}
\mathbb{P}(LC_n(\rho,\pi) \leq x)=\mathbb{P}(LI_n(\pi) \leq x).
\end{equation}

\end{proposition}
\begin{proof}
To begin the proof of (i), let $\pi'\in \frak S_n$ be such that
$\pi_i'=\rho_i$. Then, $\pi'':=\pi \pi'$ is still a uniform random
permutation of $[n]$, and so 
\begin{eqnarray*}
&&LC_n((\rho_1,\rho_2,\dots,\rho_n),(\pi_1,\pi_2,\dots,\pi_n)) \\
&&\quad \quad\quad\quad =_d  LC_n((\rho_1,\rho_2,\dots,\rho_n),
(\pi''_1,\pi''_2,\dots,\pi''_n))  \\
&&\quad\quad\quad\quad =LC_n((\rho_1,\rho_2,\dots,\rho_n),
(\pi_{\rho_1},\pi_{\rho_2},\dots,\pi_{\rho_n})),
\end{eqnarray*}
where for the second equality we used $\pi''_{i}= \pi \pi_i'=
\pi_{\rho_i}$.  Clearly,
$$\!LC_n((\rho_1,\rho_2,\dots,\rho_n),(\pi_{\rho_1},
\pi_{\rho_2},\dots,\pi_{\rho_n}))\!=_d\!
LC_n((1,2,\ldots,n),(\pi_1,\pi_2,\dots,\pi_n)),$$ and so
\eqref{lislcs1} finishes the proof of (i).

Now, for (ii), 
\begin{eqnarray*}
\mathbb{P}(LC_n(\rho,\pi) \leq x) &=& \sum_{\gamma \in
S_n} \mathbb{P}(LC_n(\gamma,\pi)
\le x | \rho=\gamma)\mathbb{P}(\rho=\gamma) \\
&=& \frac{1}{n!} \sum_{\gamma \in S_n} \mathbb{P}(LC_n((\gamma_1,\ldots,\gamma_n), 
(\pi_1,\ldots,\pi_n)) \le x) \\
&=& \frac{1}{n!} \sum_{\gamma \in S_n}
\mathbb{P}(LI_n(\pi) \leq x) \\
&=& \mathbb{P}(LI_n(\pi) \leq x), 
\end{eqnarray*}
where the third equality follows from \eqref{lislcs2}.   This proves
(ii).
\end{proof}

Clearly, the identity \eqref{eq:distributionstep}, which, in fact, is
easily seen to remain true if $\rho$ is a random permutation in
$\frak S_n$ with an arbitrary distribution, shows that the probabilistic
behavior of $LC_n(\rho, \pi)$ is identical to the probabilistic
behavior of $LI_n(\pi)$.   Among the many results on $LI_n(\pi)$ presented in
Romik~\cite{Rom}, the mean asymptotic result of Vershik and Kerov~\cite{VK}, and Logan and Shepp \cite{LS} thus implies 
that (is equivalent to): 
$$\lim_{n\to +\infty} \frac{\mathbb{E}LC_n(\rho,\pi)}{2\sqrt n} = 1.$$
Moreover, the distributional asymptotic result of Baik, Deift and
Johansson~\cite{BDJ} implies that (is equivalent to), as $n\to +\infty$,
$$\frac{LC_n(\rho, \pi) - 2 \sqrt{n}}{n^{1/6}}
  \longrightarrow F_2, \qquad \; \text{in distribution,}$$
where $F_2$ is the Tracy-Widom distribution whose cdf is given by
$$F_2(t) = \exp\left(-\int_t^{\infty} (x-t) u^2(x) dx \right),$$ where
$u$ is the solution to the Painlev\'e II equation:
$$u_{xx}=2u^3+xu \qquad \text{with} \qquad u(x) \sim Ai(x)\quad
\text{as} \quad x \rightarrow \infty.$$

To finish, let us list a few venues for future research that we find
of potential interest.

\begin{remark}\label{rmk:concluding}
(i) First, the methods of the present paper can also be used to
study sequence comparison with a general scoring functions $S$.
Namely, $S : \mathcal{A}_m \times \mathcal{A}_m \rightarrow
\mathbb{R}^+$ assigns a score to each pair of letters (the LCS
corresponds to the special case where $S(a,b) =1$ for $a=b$ and
$S(a,b) =0$ for $a \ne b$). This requires more work, but is
possible, and is presented in a separate publication (see
\cite{GHI}), where multiple words are also tackled.  
Such a result requires, at first, 
to use variance estimates, generalizing \cite{HM}, as stated in the
concluding remarks of \cite{HMa} and then to extend to higher
dimensions the closeness to the diagonal results obtained in \cite{HMdiag}.

(ii)  Challenging, is the the loss of independence both between
and inside the sequences and the loss of identical distributions 
both within and between the sequences.  Results for this type of
frameworks will also be presented elsewhere. Already for hidden Markov models (HMM), convergence results, with rates, 
are obtained for $\mathbb{E}LC_n/n$ in \cite{HK1}, while \cite{HK2} 
shows how to transfer iid normal approximation 
results such as Theorem~\ref{chat1} to the HMM case.  

(iii)  It would, similarly, also be of interest to study the random permutations
versions of (i) and (ii) above.  As in the previous section, and as far as the multiple 
sequences framework is concerned,  
the study of the length of the longest common subsequences reduces to the study of the 
length of the longest common and increasing subsequences with one 
less sequence, e.g., see \cite{HX}.  

\end{remark}

\textbf{Acknowledgements.} CH's research was supported in part by a Simons
Foundation Fellowship, grant \#267336 and the grants \#246283 and \#524678 from the
Simons Foundation.  Many thanks to the LPMA
of the Universit\'e Pierre et Marie Curie, to CIMAT and to Bogazici University for their
hospitality while part of this research was carried
out. \"{U}I is grateful to L. Goldstein
for introducing him to Stein's method and, in particular, to
Chatterjee's normal approximation results.  Also, many thanks to the
LPMA of the Universit\'e Pierre et Marie Curie for its hospitality
while part of this research was carried
out as well as to the School of Mathematics of the Georgia 
Institute of Technology while being a Hale postdoctoral Fellow.

Both authors would like to thank
the French Scientific Attach\'es Fabien Agen\`es (Los Angeles) and Nicolas Florsch (Atlanta) 
for their consular help.\! Without them, this research might not
have existed.   Lastly many thanks to Ruoting Gong, George Kerchev, Chen Xu 
and referees for their detailed 
reading and many comments which have led 
to numerous improvements on this manuscript.

\end{document}